\documentclass[a4paper,fleqn]{cas-sc}

\usepackage{subfigure}
\usepackage{float}
\usepackage{xspace}
\usepackage{enumitem}
\usepackage{makecell}
\usepackage{algorithm}
\usepackage{algpseudocode}

\usepackage[semicolon,authoryear]{natbib}

\def\tsc#1{\csdef{#1}{\textsc{\lowercase{#1}}\xspace}}
\tsc{WGM}
\tsc{QE}
\newcommand{\N}{\mathbb{N}}		
\newcommand{\R}{\mathbb{R}}		

\newcommand{\fonc}[3]{#1:  #2  \rightarrow  #3}					
\newcommand{\syst}[1]{\left \{ \begin{array}{l} #1 \end{array} \right. \kern-\nulldelimiterspace}	

\newcommand{\argmin}{\text{\normalfont argmin}}
\newcommand{\argmax}{\text{\normalfont argmax}}
\newcommand{\dom}{\text{\normalfont dom}\,}

\renewcommand{\int}{\text{\normalfont int}\,}

\newcommand{\argmind}[2]{\ensuremath{\underset{\substack{{#1}}}%
		{\mathrm{argmin}}\;\;#2 }}

\newcommand{\dist}{\text{dist}}
\makeatletter
\newlength{\algorithmboxrule}
\makeatother

\newtheorem{proposition}{Proposition}
\newtheorem{corollary}{Corollary}
\newtheorem{lemma}{Lemma}
\newtheorem{theorem}{Theorem}
\newtheorem{definition}{Definition}
\newtheorem{remark}{Remark}
\newtheorem{example}{Example}

\newenvironment{proof}[1][Proof]{\noindent\textbf{#1.} }{\ \rule{0.5em}{0.5em}}
\def\qed{\hbox to 0pt{}\hfill$\rlap{$\sqcap$}\sqcup$}
\makeatother
\begin{document}
\let\WriteBookmarks\relax
\def\floatpagepagefraction{1}
\def\textpagefraction{.001}

\shorttitle{Primal-Dual Algorithm}    

\shortauthors{}  

\title [mode = title]{Primal-dual algorithm for weakly convex functions under sharpness conditions}  



%

\author[1,2]{Ewa M. Bednarczuk}

\fnmark[1]

\ead{Ewa.Bednarczuk@ibspan.waw.pl}



\affiliation[1]{organization={Systems Research Institute, Polish Academy of Sciences},
            addressline={Newelska 6}, 
            city={Warsaw},
            postcode={01–447}, 
            country={Poland}}
\author[1]{The Hung Tran}
        \cormark[1]
        \fnmark[2]
        \ead{tthung@ibspan.waw.pl}
        
        %
        

\author[2]{Monika Syga}
\fnmark[3]
\ead{monika.syga@pw.edu.pl}



\affiliation[2]{organization={Warsaw University of Technology, Faculty of Mathematics and Information Science},
	addressline={Koszykowa 75}, 
	city={Warsaw},
	postcode={00--662}, 
	country={Poland}}

\cortext[1]{Corresponding author}



\begin{abstract}
We investigate the convergence of the primal-dual algorithm for composite optimization problems when the objective functions are weakly convex. We introduce a modified duality gap function, which is a lower bound of the standard duality gap function. Under the sharpness condition of this new function, we identify the area around the set of saddle points where we obtain the convergence of the primal-dual algorithm. We give numerical examples and applications in image denoising and deblurring to demonstrate our results.    
\end{abstract}




\maketitle
\section{Introduction}

%
\noindent Our interest is the following saddle point problem, 
\begin{equation}
	\label{prob: saddle point primal} 
	\min_{x\in X}\max_{y\in Y} f(x) +\langle Lx,y\rangle - g^*(y),
\end{equation}
where $X$ and $Y$ are Hilbert, $f:X\to\left(-\infty,+\infty\right]$ is proper lsc weakly convex and $g:Y\to\left(-\infty,+\infty\right]$ is proper lsc convex, $L:X\to Y$ is a bounded linear operator and the conjugate $g^*(y) = \sup_{w\in Y} \{\langle y,w\rangle - g(w)\}$. When $g$ is convex, problem \eqref{prob: saddle point primal} is equivalent to the composite optimization problem in Hilbert spaces \citep{Bed2020},
\begin{equation}
	\label{prob: composite primal}
	\inf_{x\in X}f\left(x\right)+g\left(Lx\right).
\end{equation}
Problem \eqref{prob: composite primal} encompasses many optimization problems and applications by using appropriate functions in the modeling process. For example, in the convex case, one can use indicator functions to model constrained problems or use total variation for image processing problems, e.g. \nolinebreak\cite{rudin1992nonlinear,chambolle2011first}.



When $f$ and $g$ are convex functions, there have been many primal-dual methods and convergence discussions of problem \eqref{prob: composite primal} for the last two decades  \nolinebreak\citep{chambolle2011first,condat2013primal,he2012convergence,vu2013splitting}. 
For nonconvex settings, there are several attempts to analyze the primal-dual algorithms for saddle point problem \eqref{prob: saddle point primal}.
For example, in the series of work \nolinebreak\citep{clason2021primal,etna_vol52_pp509-552,valkonen2014primal}, the authors consider a smooth nonlinear coupling term instead of $\langle Lx,y\rangle$ in problem \eqref{prob: saddle point primal}, while in \nolinebreak\cite{hamedani2021primal}, an accelerated primal-dual algorithm with backtracking is proposed to solve the nonlinear saddle point problem.
On the other hand, in \nolinebreak\citep{guo2023preconditioned,li2022nonsmooth,sun2018precompact,zhu2024first}, the convergence of primal-dual methods for nonconvex objective functions is investigated by relying on Kurdyka-{\L}ojasiewicz (K{\L}) inequality, which is satisfied by a large class of functions \nolinebreak\citep{bolte2007lojasiewicz}. 
Note that K{\L} inequality is equivalent to other conditions such as quadratic growth condition, metric sub-regularity condition, or error bound condition, e.g. \nolinebreak\citep{drusvyatskiy2018error,cuong2022error,bai2022equivalence}. These conditions are usually used to improve numerical algorithms' convergence rate or achieve convergence in the nonconvex setting. 

Our main contribution is the convergence analysis of primal-dual algorithms, i.e., Algorithm \ref{alg: prima-dual no xn-1} and \ref{alg: prima-dual no with primal update 1st} for solving problem \eqref{prob: saddle point primal} under the assumption of sharpness or error-bound condition of the new duality gap function (Definition \ref{def: sharpness inf-Lagrangian}) and sharpness of objective function f problem \eqref{prob: composite primal} (Definition \ref{def: sharpness primal composite}). 
Our problem setting refers to the case where the objective functions are weakly convex. It is a large class of functions, including convex and differentiable functions with Lipschitz continuous gradient. It has been shown that for weakly convex functions, the proximal subgradient is well-defined and global. In \nolinebreak\cite{jourani1996subdifferentiability}, such property is called globalization property. Recent results (\nolinebreak\cite{bednarczuk2023calculus}) show that the proximal subgradient enjoys the calculus sum rule, so we can analyze the primal-dual algorithm in terms of the proximal subgradient. 
From the works of \nolinebreak\cite{davis2018subgradient,bednarczuk2023convergence}, weakly convex functions with sharpness condition give us a linear convergence as well as specify a starting area where the convergence is guaranteed. 
Moreover, it was observed in \nolinebreak\cite{mollenhoff2015primal,shen2018nonconvex,liu2019doa} that by including a weakly convex regularizer into the model, one can improve the performance of the model. 

Sharpness conditions are usually imposed on the objective function, see, e.g., \cite{davis2018subgradient,bednarczuk2023convergence,johnstone2020faster}. For the saddle point problem, the role of objective is played by the Lagrangian and  duality gap functions  appear in investigating the convergence rate of primal-dual methods. 
However, duality gap functions are not often used to establish convergence for primal-dual algorithms. It is effective when the domains of the saddle point problem are bounded. 
On the other hand, one uses the gradient/subgradient to measure the distance to criticality by using the KKT condition as a criterion (\nolinebreak\cite[Example 11.41]{rockafellar2009variational}). 
Recently, in \nolinebreak\cite{lu2022infimal}, the authors introduce the infimal-subdifferential size, which uses subgradient to measure the distance of the iterates toward the saddle points for the primal-dual algorithm.

For our setting, when problem \eqref{prob: saddle point primal} is defined on general Hilbert spaces $X$ and $Y$, the standard duality gap function \eqref{eq: def sharpness Lagrangian inseparable} is not a good choice for investigating convergence of primal-dual algorithm. Observe that in \cite{chambolle2011first}, the restricted duality gap function is introduced to achieve convergence. 
Instead, we introduce a modified gap function (see Definition \ref{def: sharpness inf-Lagrangian} below), which is a lower bound of the standard duality gap function, to formulate our main requirement for convergence. 
This function comes naturally from the structure of the primal-dual algorithm, albeit more restrictive than the standard gap function. Combined with the sharpness condition, we use this new modified gap function for obtaining the convergence to the set of saddle points for problems with unbounded saddle point sets~(Theorem \ref{theo: pd-yx seq convergence} and \ref{theo: pd-xy seq convergence}).

The structure of the paper is as follows. 
In Section \ref{sec:preliminaries}, we give the definition of weakly convex functions, and their properties, together with proximal subgradient and calculus sum rule for weakly convex functions. In Section \ref{sec: pd-preliminary}, we introduce the Lagrange saddle point problem as well as the standard duality gap function and the new modified duality gap function. 
We also give definitions of sharpness for the duality gap functions and supporting examples. 
Section \ref{sec: primal dual yx} and \ref{sec: primal dual xy} are devoted to the convergence analysis of the primal-dual Algorithm \ref{alg: prima-dual no xn-1} and \ref{alg: prima-dual no xn-1} which can be considered as a weakly convex counterpart of the Chambolle-Pock algorithm (\cite{chambolle2011first}). 
We show that, under sharpness condition (Definition \ref{def: sharpness inf-Lagrangian}), both primal and dual iterates converge to a saddle point provided that the iterates start sufficiently close to the set of saddle points. 
Since Algorithm \ref{alg: prima-dual no xn-1}, is not symmetric with respect to the primal and dual updates, we formulate Algorithm \ref{alg: prima-dual no with primal update 1st} by changing the order of updates, and investigate its convergence in Theorem \ref{theo: pd-xy seq convergence}.
Section \ref{sec: primal sharp only} is dedicated to the analysis of the primal-dual algorithm when only problem \eqref{prob: composite primal} is sharp. Then, the dual iterate provides auxiliary information on how close we are to the set of minimizers of problem \eqref{prob: composite primal}.
Finally, in Section \ref{sec: lsc-pd numerical}, we perform numerical tests for our theory on Algorithm \ref{alg: prima-dual no xn-1} and \ref{alg: prima-dual no with primal update 1st}. We test against the $\ell_1$-regularization problem with a fully weakly convex problem. We also extend our model to image denoising and deblurring problems.

\section{Preliminaries}
\label{sec:preliminaries}
Let $X$ be Hilbert space with the inner product $\langle ,\rangle$ and the norm $\Vert x\Vert =\sqrt{\langle x,x\rangle }$.
A function $f:X\to \left( -\infty,+\infty \right]$ is proper if its domain, denoted by $\text{dom }f=\left\{ x\in X:f\left(x\right)<+\infty\right\}$ is nonempty.
We start with the  definition of $\rho$-weak convexity (also known as $\rho$-paraconvexity, see \nolinebreak\cite{jourani1996subdifferentiability,Rolewicz1979paraconvex} or $\rho$-semiconvexity, see \nolinebreak\cite{cannarsa2004semiconcave}).

\begin{definition}[\textit{$\rho$-weak convexity}]\label{def:paraconvexity}
	Let $X$ be a Hilbert space. A function $\fonc{h}{X}{(-\infty,+\infty]}$ is said to be $\rho$-weakly convex if there exists a constant $\rho\ge0$ such that for $\lambda\in[0,1]$ the following inequality holds:
	\begin{flalign}
		\quad (\forall (x,y)\in X^2) && h(\lambda x + (1-\lambda)y) \leq  \lambda h(x) + (1-\lambda) h(y) + \lambda(1-\lambda)\frac{\rho}{2}\|x-y\|^{2}. &&&
	\end{flalign}
	We refer to $\rho$ as the \textit{modulus of weak convexity} of the function $h$. 
\end{definition}
Equivalently in Hilbert space, a $\rho$-weak convexity $f$ means that $f+\frac{\rho}{2}\Vert \cdot \Vert^2$ is convex (see \nolinebreak\cite[Proposition 1.1.3]{cannarsa2004semiconcave}). A function of the form $f(x) = -a\Vert x\Vert^2$ is $2a$-weakly convex for $a>0$. Below is another example of a weakly convex function.
\begin{example}\label{ex: weak convexity}
	Let $\fonc{f}{\R}{(-\infty,+\infty)}$ be defined as $f(x) = |(x-a)(x-b)|$, where $(a,b)\in\{(a,b)\in \R^2\,|\,a<b\}$. Function $f$ is $2$-weakly convex (\cite[Example 1]{bednarczuk2023convergence}).
\end{example}
Function $f$ from \autoref{ex: weak convexity} is a simple example of a non-smooth weakly convex function. Despite its simplicity, it can be used as a modeling function in various applications, as we will show in the numerical section. 

We also give a definition of a convex conjugate of a proper function $F:X\to (-\infty,+\infty]$
\begin{definition}
	Let $X$ be Hilbert, $F:X\to (-\infty,+\infty]$ be proper, its conjugate $F^*: X \to (-\infty,+\infty]$ is defined as 
	\[
	F^* (u) = \sup_{x\in X} \ \langle u,x\rangle -F(x).
	\]
\end{definition}

For subdifferentials, since the weakly convex function is general nonconvex, its Moreau subdifferentials can be empty (e.g. Example \eqref{ex: weak convexity} for $a <x <b$). However, its proximal subgradient is nonempty, giving us an important global property of a weakly convex function. We introduce the proximal subgradient below.

\begin{definition}[Global proximal ${\varepsilon\text{-subdifferential}}$]
	\label{def: subgrad weakly convex rho}
	Let $\varepsilon\ge 0$. Let $X$ be a Hilbert space. The global proximal ${\varepsilon\text{-subdifferential}}$ of a function $F: X \to (-\infty,+\infty]$ at $x_{0}\in \text{dom } F$ for $C \geq 0$ is defined as follows:
	\begin{equation}
		\label{eq:def_prox_subdiff_eps}
		\partial^{\,\varepsilon}_{C} F (x_0) = \left\{v\in X\,|\,  F(x)-F(x_{0})\ge\langle v\ |\  x-x_{0}\rangle-C\|x-x_{0}\|^{2}-\varepsilon\ \ \forall x\in X\right\}.
	\end{equation}
	For $\varepsilon=0$, we have
	\begin{equation}
		\label{eq:def_prox_subdiff_0}
		\partial_{C} F(x_0) = \{v\in X \,|\, F(x)\geq F(x_0) + \langle v\ |\ x-x_0\rangle  - C\|x-x_0\|^{2},\;\forall x\in X\}.
	\end{equation}
	
	In view of \eqref{eq:def_prox_subdiff_0}, $\partial_{0} F(x)$ denotes the subdifferential in the sense of convex analysis. The elements of $\partial^{\,\varepsilon}_{C} F(x)$ are called proximal $\varepsilon$-subgradients.
\end{definition}

By {\cite[Proposition 3.1]{jourani1996subdifferentiability}}, for a proper lsc $\rho$-weakly convex $F$, the global proximal subdifferential $\partial_{{\rho}/{2}} F(x_0)$ defined by \eqref{eq:def_prox_subdiff_0} coincides with the set of local proximal subgradients which satisfy \eqref{eq:def_prox_subdiff_0}  locally in a neighbourhood of $x_{0}$. {This important property is called {\em the globalisation property},  see \emph{e.g.}\ [56].} It also coincides with Clarke subdifferential $\partial_{c}h(x_{0})$ (see \nolinebreak\cite{jourani1996subdifferentiability,Atenas2023Unified}). 
In this work, we only consider proper lsc $\rho$-weakly convex functions, we will use $\partial_{\rho} F(x_{0})$ instead of $\partial_{\rho/2} F(x_{0})$ for proximal subdifferentials of $F$ at $x_0$. The next proposition describes the domain of proximal subdifferentials for weakly convex function $F$.

\begin{proposition}[{\cite[Proposition 2]{bednarczuk2023calculus}}]
	Let $X$ be a Hilbert space, and $F$ is proper lsc $\rho$-weakly convex on $X$ with $\rho\geq 0$. Then 
	$
		\dom \partial_{\rho} F \subset \dom F,
	$
	and for every $\varepsilon> 0$ {we have the equality $\dom \partial_{\rho}^{\,\varepsilon} h = \dom h$.}
\end{proposition}
Moreover, we also have the sum rule for proximal subdifferentials for weakly convex functions. This is an important property when analyzing primal-dual algorithm below. 
\begin{theorem}\cite[Theorem 2]{bednarczuk2023calculus}
	\label{theo: MReps} Let $X$ be a Hilbert space. 
	For $i=0,1$, let function $\fonc{f_i}{X}{(-\infty,+\infty]}$ be proper lower semicontinuous and $\rho_i$-weakly convex on $X$ with $\rho_i\geq 0$. Then, for all $x \in \dom f_0 \cap \dom f_1$ and for all $\varepsilon_0,\varepsilon_1 \geq 0$ we have
	\begin{equation}
		\label{eq: inclusion MR: eps}
		\partial_{\varepsilon_0, \rho_0 } f_0(x) +\partial_{\varepsilon_1, \rho_1} f_1(x) \subseteq \partial_{\varepsilon,{\rho}} (f_0+f_1)(x)
	\end{equation}
	for all $\varepsilon\geq\varepsilon_0+\varepsilon_1$ and for all $\rho \geq \rho_0+\rho_1$.
	The equality
	
	\begin{equation}
		\label{eq: equality MR: eps}
		\partial_{\varepsilon, \rho_0+\rho_1} (f_0 + f_1)(x)  =
		\bigcup_{\varepsilon_{0},\varepsilon_{1}\,|\,\varepsilon_{0}+\varepsilon_{1}\leq \varepsilon} \partial_{\varepsilon_0,\rho_0} f_0(x) + \partial_{\varepsilon_1,\rho_1} f_1(x)
	\end{equation}
	holds when $\dom f_0\cap \dom f_1$ contains a point at which either $f_0+{\rho_0}/2\|\cdot \|^2$ or $f_1(\cdot)+{\rho_1}/2\|\cdot \|^2$ is continuous.
\end{theorem}

\section[Primal-Dual Weakly Convex Functions]{Primal-Dual relationship of Composite Minimization Problem with weakly convex function}
\label{sec: pd-preliminary}

In this section, we introduce the Lagrangian dual problem corresponding to \eqref{prob: composite primal} and define the modifeid duality gap function, which has a global minimizers at the saddle point of Lagrange saddle point problem. We also define sharpness condition, which will be our main tool for the convergence analysis below.

For problem \eqref{prob: composite primal}, the following holds
\begin{equation}
\label{eq: lagrange primal g weakly covnex}
\inf_{x\in X} \sup_{y\in Y} f(x) +\langle Lx,y \rangle -g^* (y) = \inf_{x\in X} f(x) +g(Lx),
\end{equation}
where the function $f:X\to (-\infty,+\infty]$ is a proper lsc weakly convex defined on Hilbert space $X$, $g:Y\to (-\infty,+\infty]$ is a proper lsc convex function on Hilbert space $Y$ and $L:X\to Y$ is a bounded linear operator with its adjoint $L^* :Y\to X$.
The LHS can be considered as  Lagrange primal problem \eqref{eq: lagrange primal g weakly covnex}. Noted that, in general, Lagrange primal is not equivalent to the primal problem \eqref{prob: composite primal} unless function $g$ is convex \nolinebreak\cite[Proposition 5.1]{Bed2020}.
The Lagrange dual problem with respect to \eqref{eq: lagrange primal g weakly covnex} is
\begin{equation}
	\label{prob: dual g convex, f weakly convex}
 \sup_{y\in Y} \inf_{x\in X} \mathcal{L}(x,y),
\end{equation}
where $\mathcal{L} :X\times Y \to (-\infty,+\infty]$ is the Lagrangian defined as
\begin{equation}
	\label{eq: lagrangians}
	\mathcal{L} (x,y) := f(x) + \langle Lx,y\rangle -g^* (y).
\end{equation}

In the next result, we show that the fully weakly convex case ($f,g$ weakly convex) in the form of problem \eqref{prob: composite primal} can be reduced to our current setting ($f$ weakly convex, $g$ convex).
\begin{lemma}
	Let $X,Y$ be Hilbert spaces, consider the functions $f:X\to (-\infty,+\infty]$ be proper lsc $\rho_f$-weakly convex and $g:Y\to (-\infty,+\infty]$ be proper lsc $\rho_g$-weakly convex, the problem \eqref{prob: composite primal}
	\begin{equation*}
		\inf_{x\in X} f(x)+g(Lx),
	\end{equation*}
	is equivalent to 
	\begin{equation}
		\label{lem: prob f wc g convex}
		\inf_{x\in X} F(x)+G(Lx),
	\end{equation}
	where $F: X\to (-\infty,+\infty]$ is proper lsc weakly convex and $G: Y\to (-\infty,+\infty]$ is proper lsc convex.
\end{lemma}
\begin{proof}
	Since $g$ is $\rho_g$-weakly convex, we can turn \eqref{prob: composite primal} into
	\[
	\inf_{x\in X} f(x) -\frac{\rho_g}{2}\Vert Lx\Vert^2 +\frac{\rho_g}{2}\Vert Lx\Vert^2 +g(Lx).
	\]
	Denoting $F(x) = f(x) -\frac{\rho_g}{2}\Vert Lx\Vert^2$ and $G(y) = g(y) +\frac{\rho_g}{2}\Vert y\Vert^2$, we obtain \eqref{lem: prob f wc g convex} where $F$ is $\rho_f+\rho_g\Vert L\Vert^2$-weakly convex (see \nolinebreak\cite[Proposition 4.5]{bednarczuk2023duality}) and the function $G$ is convex.
	The conjugate dual problem with respect to functions $F$ and $G$ is
	\begin{equation}
		\label{prob: dual G convex}
		\sup_{y\in Y}\inf_{x\in X} F\left(x\right)+\left\langle Lx,y\right\rangle -G^{*} \left(y\right),
	\end{equation}
	with the respective Lagrangian
	\begin{align}
		\label{lem: Lagrangian G convex}
		& \mathcal{L}_{FG} (x,y) = F\left(x\right)+\left\langle Lx,y\right\rangle -G^{*} \left(y\right).
	\end{align}
	Now let us consider the Lagrange primal problem with respect to \eqref{lem: Lagrangian G convex}, 
	\[
	\inf_{x\in X}\sup_{y\in Y} \mathcal{L}_{FG} (x,y) = \inf_{x\in X}\sup_{y\in Y} F\left(x\right)+\left\langle Lx,y\right\rangle -G^{*} \left(y\right) =  \inf_{x\in X} F\left(x\right)+ G^{**} \left(Lx\right).
	\]
	Because $G$ is lsc convex, $G^{**} = G$ \cite[Theorem 13.37]{Bau2011} so we finish the proof.
\end{proof}




By the fact that the conjugate function is a convex function, the Lagrangian is convex with respect to the second variable $y\in Y$. On the other hand, it is nonconvex (weakly convex) in the primal variable $x\in X$. 

Let us consider $\mathcal{K}:X\times Y\to (\infty,+\infty]$ be a proper function. Recall that $(x^*,y^*)$ is a saddle point \cite[Definition 19.16]{Bau2011} of $\mathcal{K}(x,y)$ if it satisfies
\begin{equation}
	\label{eq: lagrange saddle points prop}
	(\forall (x,y)\in X\times Y) \qquad \sup_{y\in Y}\mathcal{K} (x^*,y) \geq \mathcal{K} (x^*,y^*) \geq \inf_{x\in X} \mathcal{K} (x,y^*).
\end{equation}

Combining with the general minimax relationship i.e.
\begin{equation}
	\label{eq: saddle poitn lagrange}
\inf_{x\in X} \sup_{\hat{y}\in Y}\mathcal{K}\left(x,\hat{y}\right)- \sup_{y\in Y} \inf_{\hat{x}\in X} \mathcal{K}\left(\hat{x},y\right)
	\geq 0,
\end{equation}
one can define a gap function $\mathcal{G}_\mathcal{K}: X\times Y \to [0,\infty)$
\begin{align}
    \label{eq: gap function general}
	\mathcal{G}_\mathcal{K} (x,y) & := \sup_{(\overline{x},\overline{y})\in X\times Y} \mathcal{K}(x,\overline{y}) - \mathcal{K}(\overline{x},y).
\end{align}
This is a standard definition of duality gap function which comes from variation inequality, first appearing in \citep{auslender1973} and is widely used for general saddle point problem \citep{chambolle2011first,tran2014primal,davis2015convergence}.
From \eqref{eq: gap function general}, it is clear that $\mathcal{G}_{\mathcal{K}} (x,y)=0$ if and only if $(x,y)$ is a saddle point $(x,y)$. For the Lagrangian \eqref{eq: lagrangians}, the gap function can be written as
\begin{align}
    \label{eq: def sharpness Lagrangian inseparable}
	\mathcal{G}_{\mathcal{L}} (x,y) & = f(x)+g^*(y) + \sup_{(\hat{x},\hat{y}) \in X\times Y} \langle Lx,\hat{y}\rangle-\langle L\hat{x},y \rangle - f(\hat{x}) - g^*(\hat{y})\\
    & = f(x)+g^*(y) + g^{**}(Lx) -f^*(-L^*y). \nonumber
\end{align}

Observe that $\mathcal{G}_\mathcal{L} (x,y)$ can take $+\infty$ when $f(\hat{x}) = -\infty$ or $g^*(y) = +\infty$.
Since $f$ is proper lsc weakly convex and $g$ is proper lsc convex, the gap function $\mathcal{G}_\mathcal{L} (x,y)$, according to~\eqref{eq: def sharpness Lagrangian inseparable},
is weakly convex with respect to $x$, and convex with respect to $y$. It is also jointly weakly convex with respect to $(x,y)$ thanks to its structure.
Indeed, let $\lambda\in [0,1], (x_1,y_1),(x_2,y_2) \in X\times Y$, we have
\begin{align}
	& \mathcal{G}_\mathcal{L} (x_1 +(1-\lambda)x_2, y_1 +(1-\lambda)y_2) \nonumber\\
	& = f (x_1 +(1-\lambda)x_2) +g^*(y_1 +(1-\lambda)y_2) \nonumber\\ 
	& + \sup_{(\hat{x},\hat{y}) \in X\times Y} \langle L(x_1 +(1-\lambda)x_2 ),\hat{y}\rangle-\langle L\hat{x}, y_1 +(1-\lambda)y_2 \rangle - f(\hat{x}) - g^*(\hat{y}) \nonumber\\
	& \leq f (x_1 ) +g^*(y_1 ) + (1-\lambda) (f (x_2) +g^*(y_2)) + \sup_{(\hat{x},\hat{y}) \in X\times Y} \langle Lx_1 ,\hat{y}\rangle-\langle L\hat{x}, y_1  \rangle - f(\hat{x}) - g^*(\hat{y}) \nonumber\\
	& + (1-\lambda) \sup_{(\hat{x},\hat{y}) \in X\times Y} \langle L x_2 ,\hat{y}\rangle-\langle L\hat{x}, y_2 \rangle - f(\hat{x}) - g^*(\hat{y}) + \lambda (1-\lambda)\frac{\rho_f}{2} \Vert x_1 -x_2\Vert^2 \nonumber\\
	& \leq \mathcal{G}(x_1,y_1) +(1-\lambda) \mathcal{G}(x_2,y_2) + \lambda (1-\lambda)\frac{\rho_f}{2} \Vert (x_1,y_1) -(x_2,y_2)\Vert^2_{X\times Y}, \label{eq: weak convex of Gap}
\end{align}
where we use weak convexity of $f$ with modulus $\rho_f$,  convexity of $g^*$, and the norm of the Cartesian product $X\times Y$ is defined as
$
\Vert (x,y)\Vert_{X\times Y} = \sqrt{\Vert x \Vert_X^2 +\Vert y\Vert_Y^2}.
$
Hence, the standard duality gap function $\mathcal{G}_\mathcal{L}$ is jointly weakly convex with modulus $\rho_f$.




\subsection{Sharpness}
Let us introduce the sharpness condition for the Lagrangian, which measures how far we are from the saddle point in terms of Lagrangian values. 

\begin{definition}[Sharpness]
	\label{def: sharpness Lagrangian nonseparable}
	Let us denote $z=(x,y)\in X\times Y = Z$, and let $S\subset Z$ be the nonempty set of saddle point of $\mathcal{K}: Z\to (-\infty,+\infty]$. We say that $\mathcal{K}$ is sharp if there exists a positive constant $\mu >0$ such that
	\begin{equation*}
		(\forall z \in Z)\quad \mathcal{G}_\mathcal{K} (z) \geq \mu\text{dist}\left(z,S\right),
	\end{equation*}
	where the distance function is defined as
	\begin{equation*}
		\dist (z,S) = \inf_{\bar{z} \in Z} \Vert z -\bar{z}\Vert = \inf_{(\bar{x}, \bar{y})\in S} \sqrt{\Vert x-\bar{x}\Vert^2_X +\Vert y-\bar{y}\Vert^2_Y }.
	\end{equation*}
\end{definition}


Definition \ref{def: sharpness Lagrangian nonseparable} coincides with sharpness condition of a proper function $F:X\to (-\infty,+\infty]$ with respect to its minimizer, in view of the fact $\mathrm{argmin}\  \mathcal{G}_\mathcal{K} = S$,
\[
F (x) - \inf_{\bar{x}\in X} F(\bar{x}) \geq \mu \dist (x,\arg\min F).
\]
Sharpness condition is also known as error bound condition and is often used to speed up the convergence of algorithms \nolinebreak\cite{roulet2017sharpness,colbrook2022warpd,fercoq2022quadratic}. 
In a nonconvex setting, sharpness is used to investigate the convergence of algorithm \nolinebreak\cite{davis2018subgradient,li2019incremental,chen2021distributed}. For necessary and sufficient conditions for error bound condition, the reader is referred to the works of \nolinebreak\cite{drusvyatskiy2018error,drusvyatskiy2021nonsmooth,cuong2022error}.

In the context of primal-dual problems, there are several ways of introducing sharpness conditions. 
For example, \nolinebreak\cite{applegate2023faster} introduces the normalized duality gap function and proves that it is sharp for linear programming problems. They exploit its properties to obtain a linear convergence of the primal-dual algorithm with restart.
With an analogous approach, \nolinebreak\cite{xiong2023computational} gives two new measures and condition to study the convergence of primal-dual algorithm based on normalized dual gap function, namely limiting error ratio and LP sharpness.
Conversely, \nolinebreak\cite{fercoq2022quadratic} use smoothed dual gap function with quadratic growth condition (or sharpness of order $2$) to investigate linear convergence of primal-dual hybrid gradient for convex optimization problem.
Another line of work comes from \nolinebreak\cite{colbrook2022warpd}, where the author uses sharpness condition on the primal objective of the linear inverse problem (instead of the dual gap function) and study the linear convergence of the primal-dual algorithm.
To the best of our knowledge, these are the only works in the literature that use a concept analogous to the sharpness in the sense of Definition \ref{def: sharpness Lagrangian nonseparable}.


Nevertheless, the duality gap function $\mathcal{G}_\mathcal{K}$ is not often used to achieve convergence of primal-dual algorithm, even in the fully convex case. Instead, one can obtain the convergence by utilizing the gap function restricted on bounded set which contain the saddle points \cite[Theorem 1]{chambolle2011first}.
To investigate the convergence of the primal-dual Algorithm \ref{alg: prima-dual no xn-1} and \ref{alg: prima-dual no with primal update 1st}, we introduce inf-sharpness condition below based on a modified gap function, which will be crucial in the convergence analysis, see, e.g., Theorem \ref{theo: pd-yx seq convergence} and Theorem \ref{theo: pd-xy seq convergence}.
The modified gap function $\mathcal{H}_\mathcal{K} :Z \to (-\infty,+\infty]$ is defined as
\begin{equation}
    \mathcal{H}_\mathcal{K} (x,y) := \inf_{(\bar{x},\bar{y})\in S} \left\{\mathcal{K}\left(x,\bar{y} \right) -\mathcal{K}\left(\bar{x},y\right)\right\},
\end{equation}
where $S$ is the set of saddle points of $\mathcal{K}$.
Observe that $\mathcal{H}_\mathcal{K} (x,y) \geq 0$ for any $(x,y)\in Z$ and if $(x,y)\in S$ then $\mathcal{H}_\mathcal{K} (x,y) =0$. The converse is not true, as we will see in Example~\ref{ex: inf-sharp Lagrange 4} below. Moreover, we have
\begin{equation}
\label{eq: gap function G and H}
\mathcal{G}_\mathcal{K} (x,y) \geq \sup_{(\bar{x},\bar{y})\in S} \left\{\mathcal{K}\left(x,\bar{y}\right) -\mathcal{K}\left(\bar{x},y\right)\right\} \geq \mathcal{H}_\mathcal{K} (x,y) \geq \inf_{(\bar{x},\bar{y})\in X\times Y} \left\{\mathcal{K}\left(x,\bar{y}\right) -\mathcal{K}\left(\bar{x},y\right)\right\}.
\end{equation}
Now, we are ready to introduce sharpness condition for function $\mathcal{H}_\mathcal{K}$ which we will call Inf-Sharpness.
\begin{definition}[Inf-Sharpness]
	\label{def: sharpness inf-Lagrangian}
	Let $S\subset Z$ be the nonempty set of saddle points of the $\mathcal{K}$. 
	We say that the $\mathcal{K}$ is \emph{inf-sharp} if there exists a positive constant $\mu >0$ such that,
	\begin{equation}
		\label{eq: def sharpness inf-Lagrangian}
		(\forall z\in Z) \quad \mathcal{H}_\mathcal{K} (z) \geq \mu\text{dist}\left(z,S\right). 
	\end{equation}
\end{definition}
Observe that, from \eqref{eq: gap function G and H}, if the $\mathcal{K}$ is inf-sharp, then we also have $\mathcal{K}$ sharpness in the sense of Definition \ref{def: sharpness Lagrangian nonseparable} with the same constant $\mu>0$. 

In the case of Lagrangian $\mathcal{L}$ defined in \eqref{eq: lagrangians}, if the functions $f$ and $g^*$ are strongly convex, the set of minimizers of $\mathcal{H}_\mathcal{L}$ coincides with the set of saddle points of problem \eqref{prob: composite primal}. Indeed, let $\mathcal{H}_\mathcal{L} (x,y) =0$ and $(x^*,y^*)\in S$ be the saddle point such that the infimum is attained for $\mathcal{H}_\mathcal{L} (x,y)$, we have
\[
0 =\mathcal{H}_\mathcal{L} (x,y) = \left( f(x)-f(x^*) +\langle Lx,y^*\rangle\right)+\left( g^*(y)-g^*(y^*) - \langle Lx^*,y\rangle\right) \geq \frac{\rho_f}{2}\Vert x-x^*\Vert^2 +\frac{\rho_{g^*}}{2}\Vert y-y^*\Vert^2,
\]
where $\rho_f,\rho_{g^*}>0$ are the strongly convex modulus of $f$ and $g^*$, respectively. 

In the fully weakly convex case, one can follow the approach of \cite{fercoq2022quadratic} and define the following function
\[
\mathcal{I}_\mathcal{K} (x,y) := \inf_{(x',y')\in S} \left\{ \mathcal{K}(x,y')-\mathcal{K}(x',y) +\frac{\beta}{2}\left( \Vert x-x'\Vert^2 +\Vert y-y'\Vert^2 \right) \right\} \geq \mathcal{H}_\mathcal{K} (x,y),
\]
where $\beta >0$ is the scaling parameter. Its minimizers are the saddle points of $\mathcal{K}$, and $\beta$ can be taken to be a scaling parameter. 
The use of the function $\mathcal{I}$, in the convergence analysis of Algorithm \ref{alg: prima-dual no xn-1} and \ref{alg: prima-dual no with primal update 1st}, would require more constraints on the stepsizes.

\begin{example}
	\label{ex: Lagrange inf-sharp 1}
	Let us give a simple example of function $\mathcal{K} (x,y)$ that satisfies Definition \ref{def: sharpness inf-Lagrangian}. Consider $\mathcal{K}:\mathbb{R}\times\mathbb{R}\to (-\infty,+\infty]$ of the following form,
	\begin{equation}
		\label{exm: inf-sharp 1st}
		\mathcal{K}(x,y) = |x|-|y|,
	\end{equation}
	which is convex-concave. It has a unique saddle point at $x=y=0$.
	We have
	\begin{equation*}
		\mathcal{H}_\mathcal{K} (x,y) = \mathcal{K}(x,0) - \mathcal{K}(0,y) = |x|+|y| \geq \sqrt{x^2+y^2},
	\end{equation*}
	so the function $\mathcal{K}$ is inf-sharp with a constant $1$.
\end{example}

\begin{example}
	\label{exm: inf-sharp 2nd case}
	Even though the function $\mathcal{K}$ from the previous example is inf-sharp, it does not have the bilinear term as in \eqref{eq: lagrangians}. Let us consider another example inspired by \eqref{exm: inf-sharp 1st},
	\begin{equation}
		\label{exm: inf-sharp 2nd}
		\mathcal{L}(x,y) = |x| +xy-\frac{y^2}{8},
	\end{equation}
	which has a unique saddle point at $x=y=0$ (see figure \ref{fig:plots-ex2 inf-sharp}). Moreover, the primal and dual problems corresponding to \eqref{exm: inf-sharp 2nd} are
	\begin{align*}
		\inf_{x\in \mathbb{R}} |x| +2x^2, \quad
		& \sup_{y\in\mathbb{R}} -\frac{y^2}{8}. 
	\end{align*}
	\begin{figure}
		\centering
		\includegraphics[width=.7\textwidth]{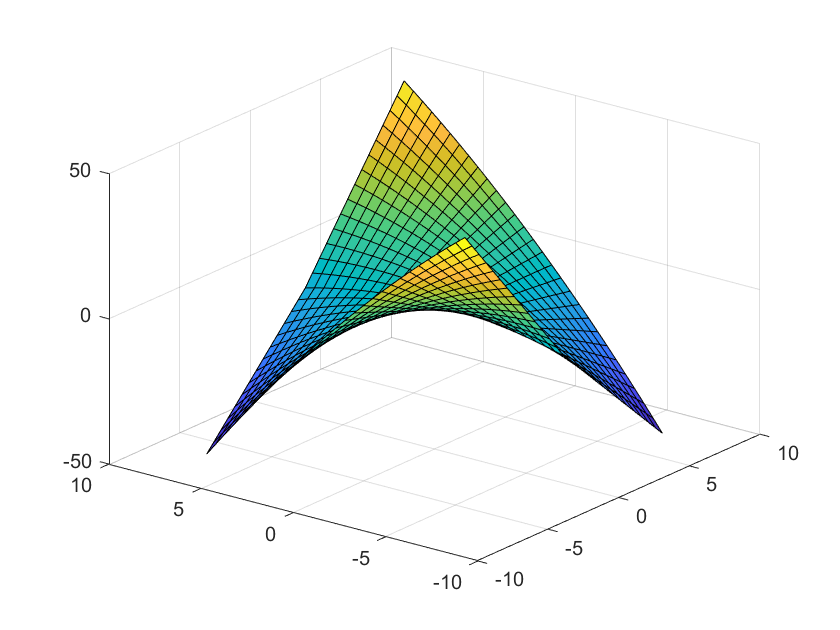}
		\caption{Lagrange function from Example \ref{exm: inf-sharp 2nd case}}
		\label{fig:plots-ex2 inf-sharp}
	\end{figure}
	
	To check for inf-sharp condition, we calculate 
	\[
	\mathcal{H}_\mathcal{L} (x,y) = \mathcal{L}(x,0) - \mathcal{L}(0,y) = |x|+\frac{y^2}{8},
	\]
	and check for the value of $\mu>0$ such that for all $(x,y)\in X\times Y$,
	\begin{equation}
		\label{eq: exm2 H- dist}
		\mathcal{H}_\mathcal{L} (x,y) - \mu \sqrt{x^2+y^2} >0.
	\end{equation}
	We test $\mu=1, 0.5, 0.1$. To see the area of interest, we plot the contour of \eqref{eq: exm2 H- dist} in Figure \ref{fig:plots-ex2 inf-sharp H-dist}.
	\begin{figure}
		\centering
			\subfigure[]{\includegraphics[width=0.32\textwidth]{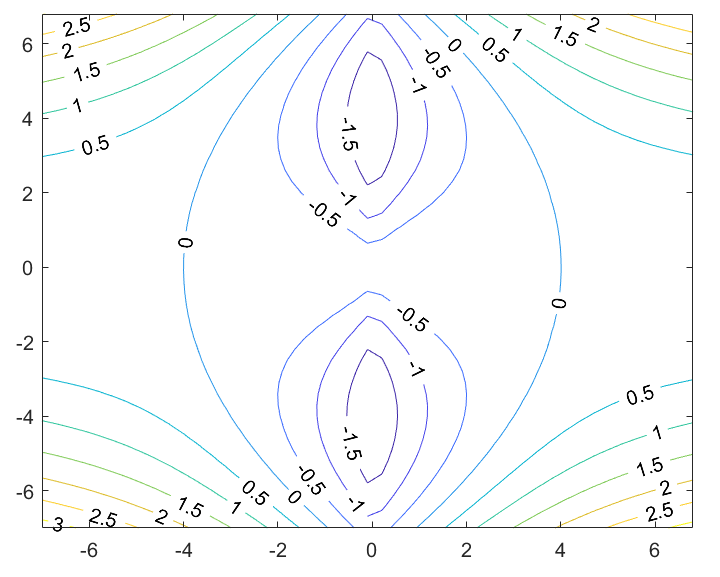}} 
			\subfigure[]{\includegraphics[width=0.32\textwidth]{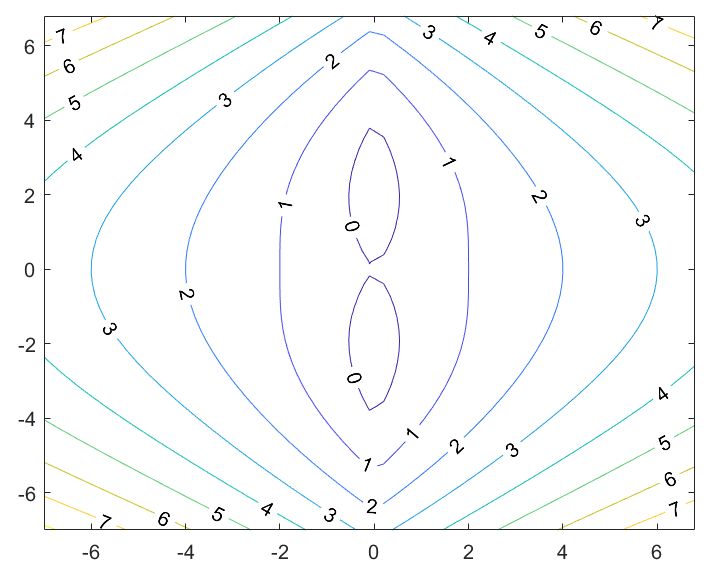}}
   \subfigure[]{\includegraphics[width=0.32\textwidth]{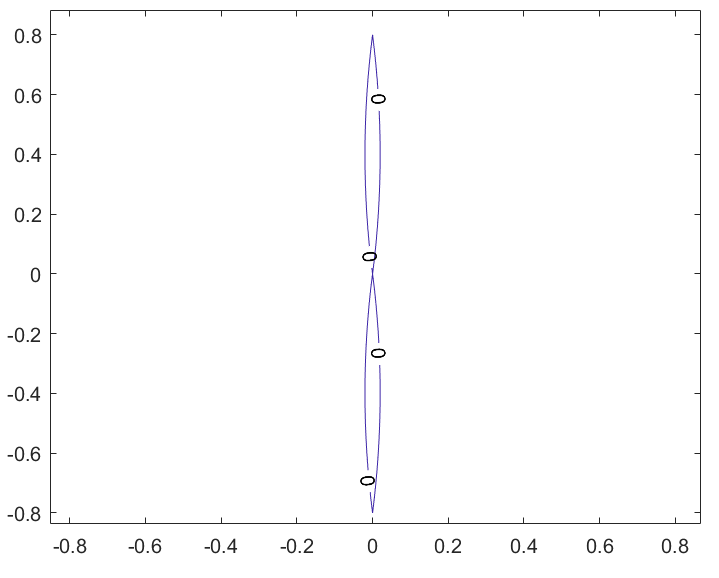}}
		\caption{Example~\ref{exm: inf-sharp 2nd case}; From left to right: contour plots of $\mathcal{H}_\mathcal{L}(x,y) - \mu \sqrt{x^2+y^2}$ at $\mu = 1, 0.5, 0.1$.}
		\label{fig:plots-ex2 inf-sharp H-dist}
	\end{figure}
	As we can see, Lagrangian is not inf-sharp around $(0,0)$ with $\mu=1$. 
	As we decrease $\mu$, we gain more area for which $\mathcal{H}_\mathcal{L} (x,y) - \mu \sqrt{x^2+y^2}>0$ around $(0,0)$ but never a full neighborhood of $(0,0)$ (even with $\mu=0.1$). Hence, the Lagrangian in this example is not inf-sharp.
\end{example}
\begin{example}
	\label{ex: Lagrange inf-sharp 3}
	Let us consider Example \ref{ex: Lagrange inf-sharp 1} with the bilinear term as
	\[
	\mathcal{L}(x,y) = |x| +xy - |y|,
	\]
	which has a saddle point at $(0,0)$. 
	The respective primal and dual problems for this Lagrangian are
	\begin{align*}
		\inf_{-1 \leq x \leq 1}  |x|, \quad
		\sup_{ -1 \leq y \leq 1} -|y|.
	\end{align*}
	
	We calculate the function $\mathcal{H}(x,y)$
	\[
	\mathcal{H}_\mathcal{L} (x,y) = \mathcal{L}(x,0) -\mathcal{L}(0,y) = |x| +|y| \geq \sqrt{x^2 +y^2} = \dist ((x,y),(0,0)).
	\]
	Therefore, the Lagrangian in this case is inf-sharp with constant $1$. 
 We also illustrate the difference $\mathcal{H}_\mathcal{L} (x,y) - \mu \dist((x,y),S$ for $\mu = 0.5, 0.9, 1$ in Figure \ref{fig:plots-ex3 inf-sharp H-dist}.
	\begin{figure}
		\centering
		\subfigure[]{\includegraphics[width=0.32\textwidth]{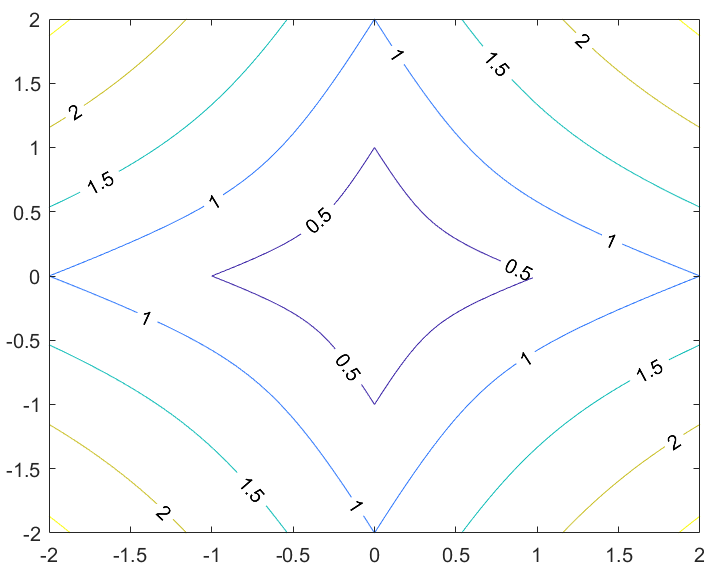}}
		\subfigure[]{\includegraphics[width=0.32\textwidth]{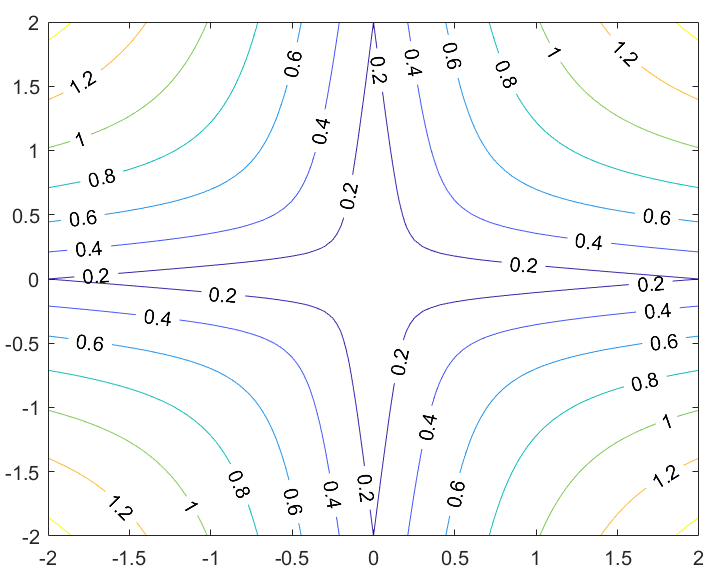}}
        \subfigure[]{\includegraphics[width=0.32\textwidth]{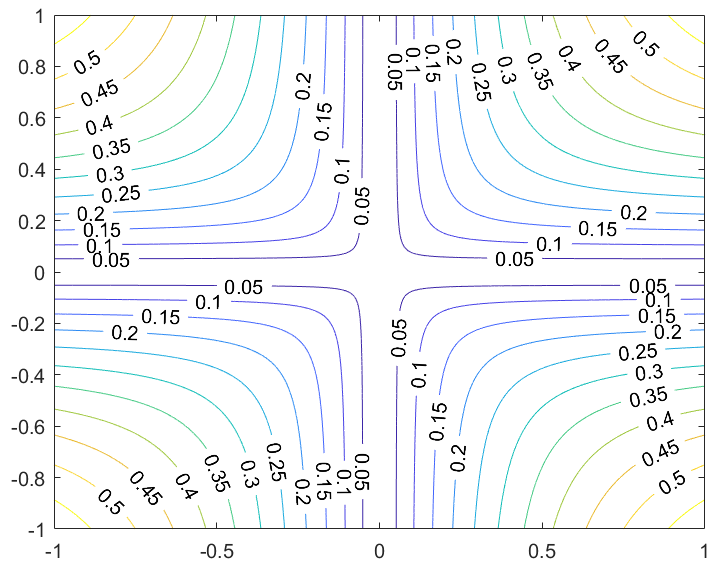}}
		\caption{Example~\ref{ex: Lagrange inf-sharp 3}; From left to right: contour plots of $\mathcal{H}_\mathcal{L}(x,y) - \mu \dist ((x,y),(0,0))$ at $\mu = 0.5, 0.9, 1$.}
		\label{fig:plots-ex3 inf-sharp H-dist}
	\end{figure}
\end{example}

\begin{example}
	\label{ex: lagrange inf-sharp 4}
	We give one more example of a more complex function. Consider the function
	\[
	\mathcal{K}(x,y) = |x|+|x^2-1| +xy - |y| - |y^2-1|,
	\]
	which has a saddle point at $(0,0)$ (see Figure \ref{fig:plots-ex4 Lagrange inf-sharp}). The function $\mathcal{K}(x,y)$ is weakly convex in $x$ and weakly concave in $y$.
	\begin{figure}
		\centering
		\includegraphics[width=0.8\textwidth, trim ={0 0 0 0}, clip]{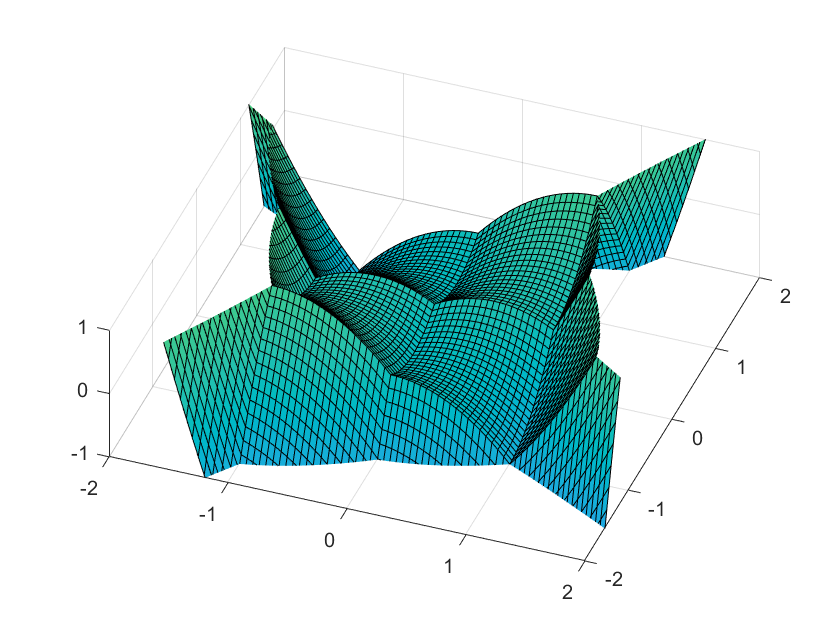}
		\caption{Function $\mathcal{K}(x,y)$ from Example \ref{ex: lagrange inf-sharp 4}.}
		\label{fig:plots-ex4 Lagrange inf-sharp}
	\end{figure}
	We calculate the function $\mathcal{H}_\mathcal{K} (x,y)$
	\begin{align*}
		\mathcal{H}_\mathcal{K} (x,y) = \mathcal{K}(x,0) -\mathcal{K}(0,y) & = |x|+|x^2-1| +|y|+|y^2-1|-2 
		\geq \sqrt{x^2 +y^2}-2 = \dist ((x,y),(0,0))-2.
	\end{align*}
	The above inequality implies that $\mathcal{K}$ is not inf-sharp with constant $1$. In fact, it is locally inf-sharp with constant $\mu<1$. We illustrate the difference $\mathcal{H}_\mathcal{K}(x,y) - \mu\dist ((x,y),(0,0))$ with $\mu = 0.5, 0.9, 1$ in Figure \ref{fig:plots-ex4 inf-sharp H-dist}.
    Notice that there are others points other than the saddle point that $\mathcal{H}_\mathcal{K} (x,y) = 0$, e.g., $\mathcal{H}_\mathcal{K} (0,1) = \mathcal{H}_\mathcal{K} (1,1) = \mathcal{H}_\mathcal{K}(1,0) = 0$.
\end{example}
\begin{figure}
	\centering
  \subfigure[]{\includegraphics[width=0.32\textwidth]{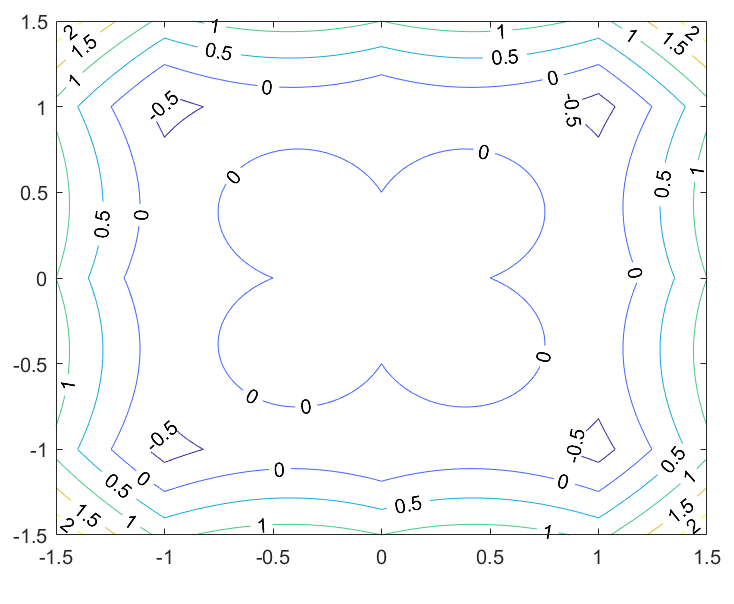}} 
	\subfigure[]{\includegraphics[width=0.32\textwidth]{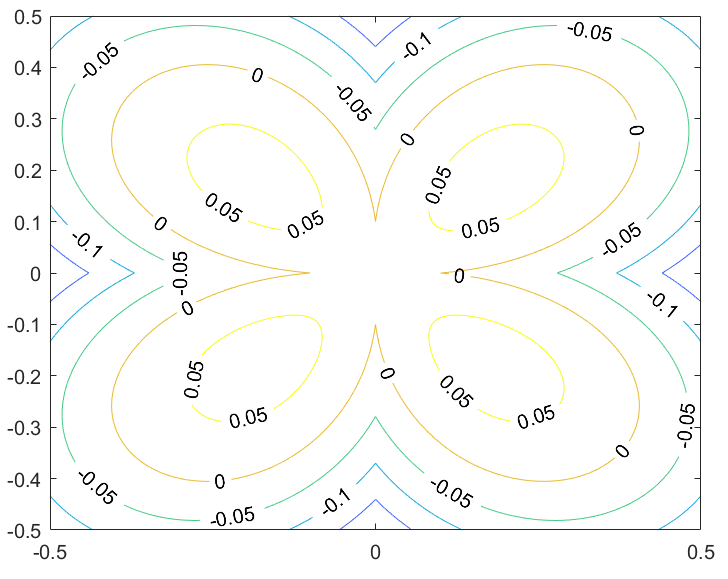}}
 \subfigure[]{\includegraphics[width=0.32\textwidth]{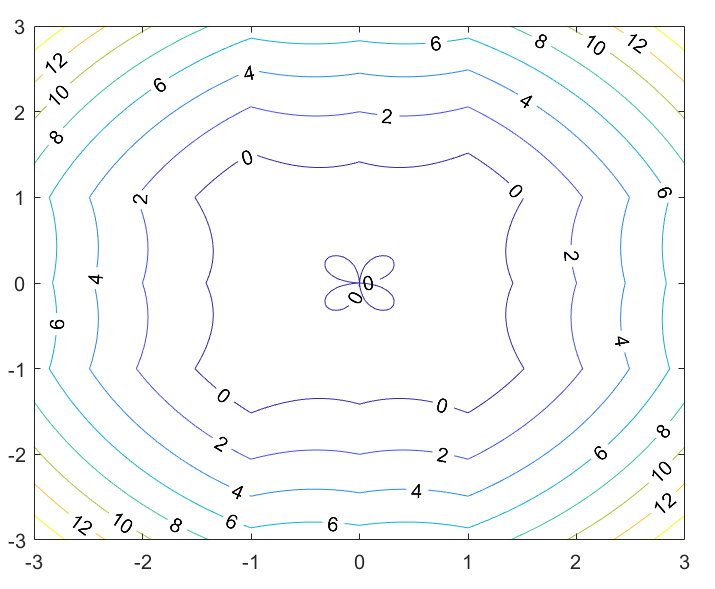}}
	\caption{Example~\ref{ex: lagrange inf-sharp 4}; From left to right: contour plots of $\mathcal{H}_\mathcal{K}(x,y) - \mu \dist ((x,y),(0,0))$ at $\mu = 0.5, 0.9, 1$.}
	\label{fig:plots-ex4 inf-sharp H-dist}
\end{figure}


\section{Primal-Dual algorithm with Dual Update}
\label{sec: primal dual yx}

Let us recall the Lagrange saddle point problem \eqref{eq: lagrange primal g weakly covnex}
\[
\inf_{x\in X}\sup_{y\in Y} f(x)+\langle Lx,y\rangle-g^*(y) = \inf_{x\in X} f(x) +g(Lx),
\]
where $f$ is proper lsc $\rho$-weakly convex function and $g$ is proper lsc convex defined on Hilbert spaces $X$ and $Y$, respectively, and $L:X\to Y$ is a bounded linear operator with the adjoint operator $L^* :Y\to X$. 
Our standing assumption in this section is that $\dom f \cap \dom g\circ L\neq \emptyset$ and the set of saddle points $S$ is nonempty.

For convex functions $f$ and $g$, the analogous of Algorithm \ref{alg: prima-dual no xn-1} has been studied in \nolinebreak\cite{chambolle2011first}. 
In \nolinebreak\cite{he2014convergence}, the authors consider Algorithm \ref{alg: prima-dual no xn-1} for convex case with no relaxation ($\theta =0$). They emphasize that the algorithm may diverge even if the Lagrangian is convex-concave. 

In our setting, under the inf-sharpness condition defined in Definition \ref{def: sharpness inf-Lagrangian},  we investigate the convergence of Algorithm \ref{alg: prima-dual no xn-1} where function $f$ is weakly convex and $g$ is convex (see Theorem \ref{theo: pd-yx seq convergence}). 
For weakly convex function $f$, we use proximal subdifferentials. It has been proved that for this particular class of functions, the proximal subdifferential can be globalized (see \nolinebreak\cite[Proposition 3.1]{jourani1996subdifferentiability}).


\begin{algorithm}
\begin{algorithmic}[1]

\State \textbf{Initialize:} $x_0 \in \dom f \cap \dom g\circ L$ and $\bar{y}_0 = y_0$
\State {\textbf{Set:} $\tau,\sigma >0, \sigma\rho <1, \theta \in [0,1], \sqrt{\sigma \tau}\Vert L\Vert<1$} 
\For {$n\in \geq 0$}
\State $y_{n+1}  =\argmin_{y\in Y} g^*(y) +\frac{1}{2\tau} \Vert y - y_n -\tau L x_n\Vert^2  $ \Comment{(dual update)} 
\State $\bar{y}_{n+1}  = y_{n+1} + \theta (y_{n+1} -y_{n}) $ \Comment{ ($\theta$ relaxation step)}
\State $x_{n+1}  =\argmin_{x\in X} f(x) +\frac{1}{2\sigma} \Vert x-x_n +\sigma L^* \bar{y}_{n+1}\Vert^2$ \Comment{(primal update)}
\EndFor
\caption{Primal-Dual Algorithm}
\label{alg: prima-dual no xn-1}
\end{algorithmic}
\end{algorithm}
\vspace{0.5cm}

In each step of Algorithm \ref{alg: prima-dual no xn-1}, we have to solve two minimization problems where the objectives are convex thanks to the appropriate choice of stepsize $\sigma\rho <1$, 
\begin{align}
	0 \in \partial \left( g^* +\frac{1}{2\tau} \Vert \cdot - y_n -\tau L x_n\Vert^2 \right) (y_{n+1}), \label{eq: conjugate sum proximal}\\
	0 \in \partial \left( f +\frac{1}{2\sigma} \Vert \cdot -x_n +\sigma L^* \bar{y}_{n+1} \Vert^2\right) (x_{n+1}),
	\label{eq: primal sum proximal}
\end{align}
where $\partial$ is the subdifferential in the sense of convex analysis.
We can apply the convex subdifferentials sum rule to the dual update \eqref{eq: conjugate sum proximal} as both functions are convex \nolinebreak\cite[Corollary 16.48]{Bau2011}. 
Equivalently, since $\sigma\rho<1$ and $f$ is weakly convex, we utilize the sum rule for proximal subdifferentials for weakly convex function in  Theorem \ref{theo: MReps} with $\varepsilon = 0$ to \eqref{eq: primal sum proximal}, we obtain
\begin{align}
	\frac{y_{n}-y_{n+1}}{\tau}+L x_{n}&\in\partial g^{*}\left(y_{n+1}\right),\\ 
	\frac{x_{n}-x_{n+1}}{\sigma}-L^{*} \bar{y}_{n+1} &\in\partial_{\rho}f\left(x_{n+1}\right). 
 \label{eq: pd prox update in subdiff}
\end{align}
Without the condition $\sigma\rho <1$, the above equivalence between \eqref{eq: primal sum proximal} and \eqref{eq: pd prox update in subdiff} do not hold.
For a detailed discussion of the splitting rule of weakly convex functions, the reader is referred to the work of \nolinebreak\cite{bednarczuk2023calculus}.

Let us estimate the Lagrangian values for the iterates generated by Algorithm \ref{alg: prima-dual no xn-1} for the Lagrangian given in \eqref{eq: lagrangians}.

\begin{proposition}
	\label{prop: pd-yx Lagrange estimation}
	Let $X,Y$ be Hilbert spaces, $f: X\to (-\infty,+\infty]$ be proper lsc $\rho$-weakly convex function and $g: Y\to (-\infty,+\infty]$ be proper lsc convex defined on Hilbert spaces $X$ and $Y$, respectively, and $L:X\to Y$ be a bounded linear operator. Let $(x_n)_{n\in\N}, (y_n)_{n\in\N}$ be the primal and dual sequences generated according to formulas in Algorithm \ref{alg: prima-dual no xn-1}. Then for any $(x,y)\in X\times Y$, we have
	\begin{align}
		& \mathcal{L}\left(x,y_{n+1}\right)-\mathcal{L}\left(x_{n+1},y\right)\nonumber\\
		\geq  & \frac{1- \sqrt{\sigma\tau}\Vert L \Vert}{2\tau} \left\Vert y-y_{n+1}\right\Vert ^{2}-\frac{1}{2\tau}\left\Vert y-y_{n}\right\Vert ^{2} + \frac{1- \theta\sqrt{\sigma\tau}\Vert L \Vert}{2\tau}\left\Vert y_{n}-y_{n+1}\right\Vert ^{2}\nonumber\\
		&+\frac{1-\sigma\rho - \theta\sqrt{\sigma\tau}\Vert L \Vert}{2\sigma} \left\Vert x-x_{n+1}\right\Vert ^{2}- \frac{1}{2\sigma}\left\Vert x-x_{n}\right\Vert ^{2}+ \frac{1-\sqrt{\sigma\tau}\Vert L \Vert}{2\sigma}\left\Vert x_{n}-x_{n+1}\right\Vert ^{2}.
		\label{eq: norm estimation xn xn+1}
	\end{align}
\end{proposition}

\begin{proof}
	From the update of Algorithm \ref{alg: prima-dual no xn-1} and \eqref{eq: pd prox update in subdiff}, by definitions of subdifferentials, we have
	\begin{align*}
		(\forall x\in X) \ f(x) -f(x_{n+1}) & \geq \frac{1}{\sigma}\langle x_n -x_{n+1}, x-x_{n+1}\rangle - \langle L^* \bar{y}_{n+1}, x-x_{n+1}\rangle -\frac{\rho}{2}\| x-x_{n+1}\|^2 ,\\
		(\forall y\in Y) \ g^* (y) -g^* (y_{n+1}) &\geq \frac{1}{\tau}\langle y_n -y_{n+1}, y-y_{n+1}\rangle + \langle L x_{n}, y-y_{n+1}\rangle,
	\end{align*}
	Summing the above inequalities and taking into account the Lagrangian given in \eqref{eq: lagrangians}, we have
	\begin{align}
		& \mathcal{L}\left(x,y_{n+1}\right)-\mathcal{L}\left(x_{n+1},y\right) \nonumber \\
		\geq & \left\langle \frac{y_{n}-y_{n+1}}{\tau},y-y_{n+1}\right\rangle +\left\langle \frac{x_{n}-x_{n+1}}{\sigma},x-x_{n+1}\right\rangle -\frac{\rho}{2}\left\Vert x-x_{n+1}\right\Vert ^{2} \nonumber\\
		& + \langle L(x_n -x_{n+1}), y-y_{n+1}\rangle+ \langle L(x-x_{n+1}), y_{n+1}-\bar{y}_{n+1}\rangle, \nonumber\\
		= & \left\langle \frac{y_{n}-y_{n+1}}{\tau},y-y_{n+1}\right\rangle +\left\langle \frac{x_{n}-x_{n+1}}{\sigma},x-x_{n+1}\right\rangle -\frac{\rho}{2}\left\Vert x-x_{n+1}\right\Vert ^{2} \nonumber\\
		& + \langle L\left( x_{n}-x_{n+1}\right),y-y_{n+1}\rangle +\theta \langle L\left( {x} -x_{n+1}\right),y_n-y_{n+1}\rangle\\
		\geq & \left\langle \frac{y_{n}-y_{n+1}}{\tau},y-y_{n+1}\right\rangle +\left\langle \frac{x_{n}-x_{n+1}}{\sigma},x-x_{n+1}\right\rangle -\frac{\rho}{2}\left\Vert x-x_{n+1}\right\Vert ^{2} \nonumber\\
		& - \sqrt{\sigma \tau}\Vert L\Vert\left( \frac{\Vert x_n-x_{n+1}\Vert^2 }{2\sigma}  + \frac{\Vert y -y_{n+1}\Vert^2}{2\tau} + \theta\frac{\Vert x-x_{n+1}\Vert^2 }{2\sigma}  + \theta \frac{\Vert y_n -y_{n+1}\Vert^2}{2\tau} \right) \nonumber\\
		= & \frac{1- \sqrt{\sigma\tau}\Vert L \Vert}{2\tau} \left\Vert y-y_{n+1}\right\Vert ^{2}-\frac{1}{2\tau}\left\Vert y-y_{n}\right\Vert ^{2} + \frac{1- \theta\sqrt{\sigma\tau}\Vert L \Vert}{2\tau}\left\Vert y_{n}-y_{n+1}\right\Vert ^{2}\nonumber\\
		&+\frac{1-\sigma\rho - \theta\sqrt{\sigma\tau}\Vert L \Vert}{2\sigma} \left\Vert x-x_{n+1}\right\Vert ^{2}- \frac{1}{2\sigma}\left\Vert x-x_{n}\right\Vert ^{2}+ \frac{1-\sqrt{\sigma\tau}\Vert L \Vert}{2\sigma}\left\Vert x_{n}-x_{n+1}\right\Vert ^{2} ,\nonumber
		\label{eq: dual-primal Lagrangian estimation}
	\end{align}
	which is \eqref{eq: norm estimation xn xn+1}.
\end{proof}

Noted that we only require $\sigma\rho <1$ in Proposition~\ref{prop: pd-yx Lagrange estimation} to solve a convex minimization sub-problem in the primal update of Algorithm~\ref{alg: prima-dual no xn-1}. 
To investigate the convergence of Algorithm~\ref{alg: prima-dual no xn-1}, we will utilize the inf-sharpness assumption and estimation \eqref{eq: norm estimation xn xn+1}.
We show that if $(x_{n},y_{n})$ are close enough to the solution set $S$, then we have decreasing property for the distance of $(x_n,y_n)$ to the set $S$.
\begin{proposition}
	\label{prop: pd-yx contraction bhvr}
	Let $X$ and $Y$ be Hilbert, $f:X\to (-\infty,+\infty]$ be a proper lsc $\rho$-weakly convex function, $g:Y\to (-\infty,+\infty]$ be proper lsc convex and $L:X\to Y$ be a bounded linear operator. Let $(x_n)_{n\in\N},(y_n)_{n\in\N}$ be the sequences generated by Algorithm \ref{alg: prima-dual no xn-1}. Let us assume that the stepsizes $\tau,\sigma$ satisfy the inequality $\sigma\rho+\theta\sqrt{\sigma\tau}\Vert L\Vert <1$, where $\theta\in[0,1]$, and the Lagrangian is inf-sharp in the sense of Definition \ref{def: sharpness inf-Lagrangian} with respect to the set of saddle point $S\neq\emptyset$. At iteration $n\in\N$, we have
	\begin{equation}
		\label{prop:eq pd-yx ditance bounds}
		\dist ((x_{n},y_{n}),S) \leq \frac{\mu }{\max\left\{ \frac{1}{2\sigma} , \frac{1}{2\tau}\right\} - A} 
		\ \Longrightarrow \
		\dist ((x_{n+1},y_{n+1}),S) \leq \frac{\mu }{\max\left\{ \frac{1}{2\sigma} , \frac{1}{2\tau}\right\} - A}, 
	\end{equation}
	where $A := \min\left\{ \frac{1- \sqrt{\sigma\tau}\Vert L \Vert}{2\tau}, \frac{1-\sigma\rho - \theta\sqrt{\sigma\tau}\Vert L \Vert}{2\sigma} \right\}$.
	Moreover,
	\begin{equation}
		\label{prop: eq pd-yx contraction}
		\dist ((x_{n},y_{n}),S) \geq \dist ((x_{n+1},y_{n+1}),S).
	\end{equation}
\end{proposition}
\begin{proof}
	Let us prove \eqref{prop: eq pd-yx contraction} first. By \eqref{eq: norm estimation xn xn+1} from Proposition \ref{prop: pd-yx Lagrange estimation}, for any $(x,y)\in X\times Y$, we have
	\begin{align}
		& \mathcal{L}\left(x,y_{n+1}\right)-\mathcal{L}\left(x_{n+1},y\right)\nonumber + \frac{1}{2\sigma}\left\Vert x-x_{n}\right\Vert ^{2}+\frac{1}{2\tau}\left\Vert y-y_{n}\right\Vert ^{2}\\
		\geq  & \frac{1- \sqrt{\sigma\tau}\Vert L \Vert}{2\tau} \left\Vert y-y_{n+1}\right\Vert ^{2} + \frac{1- \theta\sqrt{\sigma\tau}\Vert L \Vert}{2\tau}\left\Vert y_{n}-y_{n+1}\right\Vert ^{2}\nonumber\\
		&+\frac{1-\sigma\rho - \theta\sqrt{\sigma\tau}\Vert L \Vert}{2\sigma} \left\Vert x-x_{n+1}\right\Vert ^{2} + \frac{1-\sqrt{\sigma\tau}\Vert L \Vert}{2\sigma}\left\Vert x_{n}-x_{n+1}\right\Vert ^{2} \nonumber\\
		\geq & \frac{1- \sqrt{\sigma\tau}\Vert L \Vert}{2\tau} \left\Vert y-y_{n+1}\right\Vert ^{2} + \frac{1-\sigma\rho - \theta\sqrt{\sigma\tau}\Vert L \Vert}{2\sigma} \left\Vert x-x_{n+1}\right\Vert ^{2} \nonumber\\
		\geq & A \inf_{(x,y)\in S} \left\Vert x-x_{n+1}\right\Vert ^{2} +\left\Vert y-y_{n+1}\right\Vert ^{2} = A \dist^2 ((x_{n+1},y_{n+1}),S),
		\label{eq: proof Lagrange norm-1}
	\end{align}
	where $A := \min\left\{ \frac{1- \sqrt{\sigma\tau}\Vert L \Vert}{2\tau}, \frac{1-\sigma\rho - \theta\sqrt{\sigma\tau}\Vert L \Vert}{2\sigma} \right\} >0$ thanks to the assumption on the stepsizes.
	The last inequality in \eqref{eq: proof Lagrange norm-1} holds as we can take any $(x,y)\in S$.
	Rearrange both sides of \eqref{eq: proof Lagrange norm-1}, we obtain
	\begin{align}
		\frac{1}{2\sigma}\left\Vert x-x_{n}\right\Vert ^{2}+\frac{1}{2\tau}\left\Vert y-y_{n}\right\Vert ^{2}
		\geq & A \dist^2 ((x_{n+1},y_{n+1}),S) + \mathcal{L}\left(x_{n+1},y\right) - \mathcal{L}\left(x,y_{n+1}\right).
		\label{eq: proof Lagrange norm-2}
	\end{align}
	Let us take infimum on both sides of \eqref{eq: proof Lagrange norm-2} with respect to $(x,y)\in S$, 
	we obtain
	\begin{align}
		\label{eq: proof Lagrange dist-1}
		\max\left\{ \frac{1}{2\sigma}, \frac{1}{2\tau}\right\} \dist^2 ((x_n,y_n),S) 
		& \geq A \dist^2 ((x_{n+1},y_{n+1}),S) +\underbrace{\inf_{(x,y)\in S} \left\{\mathcal{L}\left(x_{n+1},y\right) -\mathcal{L}\left(x,y_{n+1}\right)\right\}}_{=\mathcal{H}_\mathcal{L} (x_{n+1},y_{n+1})}.
	\end{align}
	We apply sharpness assumption to \eqref{eq: proof Lagrange dist-1} and obtain
	\begin{equation}
		\max\left\{ \frac{1}{2\sigma}, \frac{1}{2\tau}\right\} \dist^2 ((x_n,y_n),S) \geq A \dist^2 ((x_{n+1},y_{n+1}),S) + \mu \dist ((x_{n+1},y_{n+1}),S).
		\label{eq: proof Lagrange dist-2}
	\end{equation}
	To obtain \eqref{prop: eq pd-yx contraction}, we need
	\begin{align*}
		A \dist^2 ((x_{n+1},y_{n+1}),S) + \mu \dist ((x_{n+1},y_{n+1}),S) 
		&\geq \max\left\{ \frac{1}{2\sigma}, \frac{1}{2\tau}\right\} \dist^2 ((x_{n+1},y_{n+1}),S), 
	\end{align*}
	which holds when
	\begin{equation}
		\label{eq: proof bound dn+1}
		\dist ((x_{n+1},y_{n+1}),S) \leq \frac{\mu }{\max\left\{ \frac{1}{2\sigma}, \frac{1}{2\tau}\right\} - A}.
	\end{equation}
	The quantity on RHS of \eqref{eq: proof bound dn+1} is well-defined as 
	\[
	\max\left\{ \frac{1}{2\sigma}, \frac{1}{2\tau}\right\} \geq \min\left\{ \frac{1}{2\sigma}, \frac{1}{2\tau}\right\} > \min\left\{ \frac{1- \sqrt{\sigma\tau}\Vert L \Vert}{2\tau}, \frac{1-\sigma\rho - \theta\sqrt{\sigma\tau}\Vert L \Vert}{2\sigma} \right\},
	\]
	thanks to the assumptions on the stepsizes $\tau,\sigma$. Next, we will prove that for any $n\in\N$
	\[
	\dist ((x_{n},y_{n}),S) \leq \frac{\mu }{\max\left\{ \frac{1}{2\sigma} , \frac{1}{2\tau}\right\} - A} \Rightarrow 
	\dist ((x_{n+1},y_{n+1}),S) \leq \frac{\mu }{\max\left\{ \frac{1}{2\sigma} , \frac{1}{2\tau}\right\} - A}.
	\]
	Let us start from \eqref{eq: proof Lagrange dist-2}
	\begin{align}
		\max\left\{ \frac{1}{2\sigma}, \frac{1}{2\tau}\right\} \frac{\mu^2}{ \left(\max\left\{ \frac{1}{2\sigma}, \frac{1}{2\tau}\right\} - A \right)^2} 
		&\geq \max\left\{ \frac{1}{2\sigma}, \frac{1}{2\tau}\right\} \dist^2 ((x_n,y_n),S) \nonumber\\
		& \geq A \dist^2 ((x_{n+1},y_{n+1}),S) + \mu \dist ((x_{n+1},y_{n+1}),S).
		\label{eq: proof Lagrange dist xn+1 bound-1}
	\end{align}
	Inequality \eqref{eq: proof Lagrange dist xn+1 bound-1} becomes a quadratic inequality with respect to $\dist ((x_{n+1},y_{n+1}),S)$. The quadratic equation taken from \eqref{eq: proof Lagrange dist xn+1 bound-1} has two distinct solutions, which implies that
	\[
	0 \leq \dist ((x_{n+1},y_{n+1}),S) \leq \frac{\mu }{\max\left\{ \frac{1}{2\sigma} , \frac{1}{2\tau}\right\} - A}.
	\]
\end{proof}

\begin{remark}
  When the iterates are close enough to the solution $S$, the sharpness of order $1$ implies the sharpness of order $2$, which is known as the quadratic growth condition. It has been shown in \cite[Theorem 3.1]{liao2024error} that for a weakly convex function (see  formula \eqref{eq: weak convex of Gap} for the weak convexity of the duality gap function $\mathcal{G}_\mathcal{L}$), quadratic growth condition is equivalent to subdifferential error bound and P{\L} inequality, which are well-known regularity conditions used to achieve linear convergence rate.
   On the other hand, inf-sharpness of function $\mathcal{H}_\mathcal{L}$ (Definition \ref{def: sharpness inf-Lagrangian}) implies sharpness of order $1$ of function $\mathcal{H}_\mathcal{L}$.
\end{remark}

When inequality in \eqref{prop:eq pd-yx ditance bounds} is strict, we obtain convergence of the distance function.
\begin{corollary}
	\label{cor: pd-yx dist to zero}
	Let $X$ and $Y$ be Hilbert, $f:X\to (-\infty,+\infty]$ be a proper lsc $\rho$-weakly convex function, $g:Y\to (-\infty,+\infty]$ be proper lsc convex and $L:X\to Y$ be a bounded linear operator. Let $(x_n)_{n\in\N},(y_n)_{n\in\N}$ be the sequences generated by Algorithm \ref{alg: prima-dual no xn-1}. Let us assume that $\sigma\rho+\theta\sqrt{\sigma\tau}\Vert L\Vert <1$ and the Lagrangian is inf-sharp in the sense of Definition \ref{def: sharpness inf-Lagrangian} with respect to the set of saddle points $S\neq\emptyset$. If the starting point $(x_0,y_0)$ satisfies
	\[
	\dist ((x_{0},y_{0}),S) < \frac{\mu }{\max\left\{ \frac{1}{2\sigma} , \frac{1}{2\tau}\right\} - A}, 
	\]
	where $ A = \min\left\{ \frac{1- \sqrt{\sigma\tau}\Vert L \Vert}{2\tau}, \frac{1-\sigma\rho - \theta\sqrt{\sigma\tau}\Vert L \Vert}{2\sigma} \right\}$.
	Then $\dist((x_n,y_n),S)$ tends to zero as $n\to \infty$.
\end{corollary}
\begin{proof}
	By assumption and Proposition \ref{prop: pd-yx contraction bhvr}, $\dist((x_n,y_n),S)$ is non-increasing and 
	\[
	\dist ((x_{n},y_{n}),S) < \frac{\mu }{\max\left\{ \frac{1}{2\sigma} , \frac{1}{2\tau}\right\} - A},
	\]
	for all $n\in \N$. Moreover, we have
	\begin{align}
		\left( A+\frac{\mu}{\dist((x_{0},y_{0}),S)}\right) \dist^2 ((x_{n+1},y_{n+1}),S) &
		\leq \left( A+\frac{\mu}{\dist((x_{n+1},y_{n+1}),S)}\right) \dist^2 ((x_{n+1},y_{n+1}),S) \nonumber \\
		& \leq \max\left\{ \frac{1}{2\sigma}, \frac{1}{2\tau}\right\} \dist^2 ((x_n,y_n),S).
		\label{cor: eq proof contraction with dist0}
	\end{align}
	Inequality \eqref{cor: eq proof contraction with dist0} implies that 
	\begin{equation}
		\label{cor: eq proof dn+1 < d0}
		\dist^2 ((x_{n+1},y_{n+1}),S) \leq (B)^{n-1} \dist^2 ((x_{0},y_{0}),S),
	\end{equation}
	where 
	$B := \frac{ \max\left\{ \frac{1}{2\sigma}, \frac{1}{2\tau}\right\} }{ A+\frac{\mu}{\dist((x_{0},y_{0}),S)} }.
	$
	By assumption on $\dist((x_{0},y_{0}),S)$, we have $0 < B < 1$ since
	\[
	A+\frac{\mu}{\dist((x_{0},y_{0}),S)} >  \max\left\{ \frac{1}{2\sigma}, \frac{1}{2\tau}\right\}.
	\]
	As \eqref{cor: eq proof dn+1 < d0} holds for any $n\in\N$, we let $n\to \infty$ and we infer that 
	\[
	\lim_{n\to\infty} \dist((x_{n},y_{n}),S) = 0.
	\]
\end{proof}

Next, we prove the convergence of $(x_n,y_n)_{n\in\N}$ to a saddle point in $S$.
\begin{theorem}
	\label{theo: pd-yx seq convergence}
	Let $X$ and $Y$ be Hilbert, $f:X\to (-\infty,+\infty]$ be a proper lsc $\rho$-weakly convex function, $g:Y\to (-\infty,+\infty]$ be proper lsc convex and $L:X\to Y$ be a bounded linear operator. Let $(x_n)_{n\in\N},(y_n)_{n\in\N}$ be the sequences generated by Algorithm \ref{alg: prima-dual no xn-1}. Let us assume that the stepsizes $\tau,\sigma$ satisfy $\sigma\rho+\theta\sqrt{\sigma\tau}\Vert L\Vert <1$, and the Lagrangian is inf-sharp in the sense of Definition \ref{def: sharpness inf-Lagrangian} with respect to the set of saddle point $S\neq\emptyset$. If the starting point $(x_0,y_0)$ satisfies
	\[
	\dist ((x_{0},y_{0}),S) < \frac{\mu }{\max\left\{ \frac{1}{2\sigma} , \frac{1}{2\tau}\right\} - A}, 
	\]
	where $A = \min\left\{ \frac{1- \sqrt{\sigma\tau}\Vert L \Vert}{2\tau}, \frac{1-\sigma\rho - \theta\sqrt{\sigma\tau}\Vert L \Vert}{2\sigma} \right\}.$
	Then $(x_n,y_n)$ converges to a saddle point.
\end{theorem}
\begin{proof}
	By \eqref{eq: norm estimation xn xn+1} from Proposition \ref{prop: pd-yx Lagrange estimation} and the inf-sharpness of Lagrangian, we have
	\begin{align}
		& \max\left\{ \frac{1}{2\sigma}, \frac{1}{2\tau}\right\} \dist^2 ((x_n,y_n),S)\nonumber\\
		& \geq A \dist^2 ((x_{n+1},y_{n+1}),S) + \mu \dist ((x_{n+1},y_{n+1}),S)\nonumber\\
		& + \frac{1- \theta\sqrt{\sigma\tau}\Vert L \Vert}{2\tau}\left\Vert y_{n}-y_{n+1}\right\Vert ^{2} + \frac{1-\sqrt{\sigma\tau}\Vert L \Vert}{2\sigma}\left\Vert x_{n}-x_{n+1}\right\Vert ^{2} \nonumber\\
		& \geq \min \left\{ \frac{1- \theta\sqrt{\sigma\tau}\Vert L \Vert}{2\tau} , \frac{1-\sqrt{\sigma\tau}\Vert L \Vert}{2\sigma} \right\} \Vert (x_n,y_n)- (x_{n+1},y_{n+1})\Vert^2_{X\times Y}.
		\label{theo: eq proof dist xn xn+1}
	\end{align}
	Taking square root on both sides of \eqref{theo: eq proof dist xn xn+1} and summing up till $N\in\N$, combining with the result from Corollary \ref{cor: pd-yx dist to zero}, we obtain
	\begin{align}
		C  \sum_{i=0}^N B^{\frac{i-1}{2}} & \geq \sum_{i=0}^N \Vert (x_i,y_i)- (x_{i+1},y_{i+1})\Vert_{X\times Y},
		\label{theo: eq proof sum xn xn+1}
	\end{align}
	where
	\[
	B = \frac{ \max\left\{ \frac{1}{2\sigma}, \frac{1}{2\tau}\right\} }{ A+\frac{\mu}{\dist((x_{0},y_{0}),S)} }, \quad 
	C := \sqrt{ \frac{ \max\left\{ \frac{1}{2\sigma}, \frac{1}{2\tau}\right\}}{\min \left\{ \frac{1- \theta\sqrt{\sigma\tau}\Vert L \Vert}{2\tau} , \frac{1-\sqrt{\sigma\tau}\Vert L \Vert}{2\sigma} \right\}} } \dist ((x_0,y_0),S).
	\]
	From the proof of Corollary \ref{cor: pd-yx dist to zero} we have $0< B<1$. Letting $N\to \infty$, the LHS of \eqref{theo: eq proof sum xn xn+1} is summable which implies that 
	\[
	\sum_{i} \Vert (x_i,y_i)- (x_{i+1},y_{i+1})\Vert_{X\times Y} <+\infty.
	\]
	As $X,Y$ are Hilbert, so is $X\times Y$. By {\cite[Theorem 1]{Bolte2014_Proximal}}, $(x_n,y_n)_{n\in\N}$ is Cauchy , so it converges. Combining this with the result obtained in Corollary \ref{cor: pd-yx dist to zero} which is  $\dist((x_n,y_n),S)\to 0$ as $n\to \infty$, $(x_n,y_n)_{n\in\N}$ must converge to a saddle point in $S$.
\end{proof}

\section{Primal-Dual algorithm with primal update first}
\label{sec: primal dual xy}

In this section, we consider Algorithm \ref{alg: prima-dual no with primal update 1st}, which differs from Algorithm \ref{alg: prima-dual no xn-1} by updating the primal variable $x_n$ first instead of $y_n$. We show the convergence results similar to Proposition \ref{prop: pd-yx contraction bhvr}, and Theorem \ref{theo: pd-yx seq convergence} of the previous subsection with some differences concerning the assumption on the stepsizes $\tau,\sigma >0$.

\begin{algorithm}
\begin{algorithmic}
\State	\textbf{Initialize:} $x_0 \in \dom f \cap \dom g\circ L$ and $\bar{x}_0 = x_0$
\State {\textbf{Set:} $\tau,\sigma >0, \sigma\rho <1, \theta \in [0,1], \sqrt{\sigma \tau}\Vert L\Vert<1$} \\
\For{For $n\geq 0$} 
\State $x_{n+1} =\argmin_{x\in X} f(x) +\frac{1}{2\sigma} \Vert x-x_n +\tau L^* {y}_{n}\Vert^2 $ \Comment{(primal update)}
\State $\bar{x}_{n+1} = x_{n+1} + \theta (x_{n+1} -x_{n}) $ \Comment{($\theta$ relaxation step)}
\State $y_{n+1} =\argmin_{y\in Y} g^*(y) +\frac{1}{2\tau} \Vert y - y_n -\tau L \bar{x}_{n+1} \Vert^2 $  \Comment{(dual update)}
\EndFor
\caption{Primal-Dual Algorithm with primal update first}
	\label{alg: prima-dual no with primal update 1st}
\end{algorithmic}
\end{algorithm}
\vspace{0.5cm}

Analogously to Proposition \ref{prop: pd-yx Lagrange estimation}, we start by analyzing the behavior of the Lagrangian values given in \eqref{eq: lagrangians}.

\begin{proposition}
	\label{prop: pd-xy Lagrange estimation}
	Let $X,Y$ be Hilbert spaces, $f: X\to (-\infty,+\infty]$ be proper lsc $\rho$-weakly convex function and $g: Y\to (-\infty,+\infty]$ be proper lsc convex defined on Hilbert spaces $X$ and $Y$, respectively, and $L:X\to Y$ be a bounded linear operator. Let $(x_n)_{n\in\N}, (y_n)_{n\in\N}$ be sequences generated by Algorithm \ref{alg: prima-dual no with primal update 1st}. Then for any $(x,y)\in X\times Y$, we have
	\begin{align}
		& \mathcal{L}\left(x,y_{n+1}\right)-\mathcal{L}\left(x_{n+1},y\right)\nonumber\\
		\geq  & \frac{1- \theta\sqrt{\sigma\tau}\Vert L \Vert}{2\tau} \left\Vert y-y_{n+1}\right\Vert ^{2}-\frac{1}{2\tau}\left\Vert y-y_{n}\right\Vert ^{2} + \frac{1-  \sqrt{\sigma\tau}\Vert L \Vert}{2\tau}\left\Vert y_{n}-y_{n+1}\right\Vert ^{2}\nonumber\\
		&+\frac{1-\sigma\rho - \sqrt{\sigma\tau}\Vert L \Vert}{2\sigma} \left\Vert x-x_{n+1}\right\Vert ^{2}- \frac{1}{2\sigma}\left\Vert x-x_{n}\right\Vert ^{2}+ \frac{1-\theta \sqrt{\sigma\tau}\Vert L \Vert}{2\sigma}\left\Vert x_{n}-x_{n+1}\right\Vert ^{2}.
		\label{eq: pd-xy norm estimation xn xn+1}
	\end{align}
\end{proposition}

\begin{proof}
	We proceed similarly to the proof of Proposition \ref{prop: pd-yx Lagrange estimation}.
\end{proof}

As we can see, the difference between Proposition \ref{prop: pd-yx Lagrange estimation} and Proposition \ref{prop: pd-xy Lagrange estimation} is the position of $\theta$ which appears in the term $\Vert y-y_{n+1}\Vert^2$ instead of $\Vert x-x_{n+1}\Vert^2$. Which will cause a small change in the bound of $\dist((x_n,y_n),S)$ as we see in the following result.

\begin{proposition}
	\label{prop: pd-xy contraction bhvr}
	Let $X$ and $Y$ be Hilbert, $f:X\to (-\infty,+\infty]$ be a proper lsc $\rho$-weakly convex function, $g:Y\to (-\infty,+\infty]$ be proper lsc convex and $L:X\to Y$ be a bounded linear operator. Let $(x_n)_{n\in\N},(y_n)_{n\in\N}$ be the sequences generated by Algorithm \ref{alg: prima-dual no with primal update 1st}. Let us assume that the stepsizes $\tau,\sigma$ satisfy  $\sigma\rho+\sqrt{\sigma\tau}\Vert L\Vert <1$ and the Lagrangian is inf-sharp in the sense of Definition \ref{def: sharpness inf-Lagrangian} with respect to the set of saddle points $S\neq\emptyset$. At iteration $n\in\N$, we have
	\begin{equation}
		\label{prop:eq pd-xy ditance bounds}
		\dist ((x_{n},y_{n}),S) \leq \frac{\mu }{\max\left\{ \frac{1}{2\sigma} , \frac{1}{2\tau}\right\} - A_1} \ \Longrightarrow \  
		\dist ((x_{n+1},y_{n+1}),S) \leq \frac{\mu }{\max\left\{ \frac{1}{2\sigma} , \frac{1}{2\tau}\right\} - A_1}, 
	\end{equation}
	where $A_1 := \min\left\{ \frac{1- \theta\sqrt{\sigma\tau}\Vert L \Vert}{2\tau}, \frac{1-\sigma\rho - \sqrt{\sigma\tau}\Vert L \Vert}{2\sigma} \right\}$.
	Moreover,
	\begin{equation}
		\label{prop: eq pd-xy contraction}
		\dist^2 ((x_{n},y_{n}),S) \geq \dist^2 ((x_{n+1},y_{n+1}),S).
	\end{equation}
	As a consequence, if the initials $(x_0,y_0)$,
	$
	\dist ((x_{0},y_{0}),S) < \frac{\mu }{\max\left\{ \frac{1}{2\sigma} , \frac{1}{2\tau}\right\} - A_1}.
	$
	Then $\dist((x_n,y_n),S)$ tends to zero as $n\to \infty$.
\end{proposition}
The proof of Proposition \ref{prop: pd-xy contraction bhvr} follows the same lines as the proof of Proposition \ref{prop: pd-yx contraction bhvr} and Corollary \ref{cor: pd-yx dist to zero} so we will not give it here. As $(\dist((x_n,y_n),S))_{n\in\N}$ with $(x_n)_{n\in\N},(y_n)_{n\in\N}$ generated by Algorithm \ref{alg: prima-dual no with primal update 1st} behaves in the same way as in the case of Algorithm \ref{alg: prima-dual no xn-1}, so is the convergence of $(x_n,y_n)_{n\in\N}$.

\begin{theorem}
	\label{theo: pd-xy seq convergence}
	Let $X$ and $Y$ be Hilbert, $f:X\to (-\infty,+\infty]$ be a proper lsc $\rho$-weakly convex function, $g:Y\to (-\infty,+\infty]$ be proper lsc convex and $L:X\to Y$ be a bounded linear operator. Let $(x_n)_{n\in\N},(y_n)_{n\in\N}$ be the sequences generated by Algorithm \ref{alg: prima-dual no with primal update 1st}. Let us assume that the stepsizes $\tau,\sigma$ satisfy $\sigma\rho+ \sqrt{\sigma\tau}\Vert L\Vert <1$ and the Lagrangian is inf-sharp in the sense of Definition \ref{def: sharpness inf-Lagrangian} with respect to the set of saddle point $S\neq\emptyset$. If the starting point $(x_0,y_0)$ satisfies
	\begin{equation}
		\label{eq: starting point of pd-xy}
		\dist ((x_{0},y_{0}),S) < \frac{\mu }{\max\left\{ \frac{1}{2\sigma} , \frac{1}{2\tau}\right\} - A_1},
	\end{equation}
	where $A_1 = \min\left\{ \frac{1- \theta\sqrt{\sigma\tau}\Vert L \Vert}{2\tau}, \frac{1-\sigma\rho - \sqrt{\sigma\tau}\Vert L \Vert}{2\sigma} \right\}$.
	Then $(x_n,y_n)$ converges to a saddle point.
\end{theorem}
The condition \eqref{eq: starting point of pd-xy} gives us the radius of an open ball around the set $S$ such that Algorithm \ref{alg: prima-dual no with primal update 1st} converges to a saddle point, provided that we start inside the ball. When starting outside the ball, the sequences can converge to other points but not the saddle point. We will observe this behavior in the numerical example below.

\section{Solution to the Primal Problem}
\label{sec: primal sharp only}

While the inf-sharpness in the sense of Definition \ref{def: sharpness inf-Lagrangian} helps us analyze the convergence of Algorithm \ref{alg: prima-dual no xn-1} and Algorithm \ref{alg: prima-dual no with primal update 1st}, one needs to know the set of saddle points (both the primal and dual solutions) to calculate $\mathcal{H}_\mathcal{K} (x,y)$.
It is interesting to solve problem \eqref{prob: composite primal} with only knowledge of the primal objective; hence, we ask ourselves how much we know about the convergence of Algorithm \ref{alg: prima-dual no xn-1} or \ref{alg: prima-dual no with primal update 1st}. Is it possible to assume sharpness condition only on the primal objective functions instead of on the Lagrangian, and how is it related to Definition \ref{def: sharpness Lagrangian nonseparable} and Definition \ref{def: sharpness inf-Lagrangian}? 

In this section, we consider the same setting as in Section \ref{sec: pd-preliminary}, where $f$ is $\rho$-weakly convex and $g$ is convex. We analyze the gap function $\mathcal{G}_\mathcal{L} (x,y)$ given by \eqref{eq: def sharpness Lagrangian inseparable} with the Lagrangian given in \eqref{eq: lagrangians},
\[
\mathcal{G}_\mathcal{L} (x,y) = \sup_{\overline{x}\in X,\overline{y}\in Y} \mathcal{L}(x,\overline{y}) - \mathcal{L} (\overline{x},y) = f(x)+g^{**}(Lx) + g^*(y) + \sup_{\overline{x}\in X} \{
-\langle L\overline{x},y \rangle -f(\overline{x})\}.
\]
Since $g$ is convex, we have $g=g^{**}$. Let us take
\begin{align}
	\inf_{y\in Y} \mathcal{G}_\mathcal{L} (x,y) & = \inf_{y\in Y} \left\{ (f+g\circ L)(x) + \sup_{\overline{x}\in X} 
	\{ -\langle L\overline{x},y \rangle -f(\overline{x}) + g^*(y) \}\right\} \nonumber\\
	& = (f+g\circ L)(x) -\sup_{y\in Y}\inf_{\overline{x} \in X} \left\{ f(\overline{x}) + \langle L\overline{x},y \rangle - g^*(y) \right\} \nonumber\\
	& \geq (f+g\circ L)(x) -\inf_{\overline{x} \in X} \sup_{y\in Y} \mathcal{L}(\overline{x},y) \nonumber\\
	& = (f+g\circ L)(x) -\inf_{\overline{x} \in X} (f+g\circ L)(\overline{x}),
	\label{eq: inf gap and primal}
\end{align}
where we use weak duality in the third inequality, and the last equality is the equivalence between Lagrangian primal \eqref{eq: lagrange primal g weakly covnex} and primal problem \eqref{prob: composite primal} thanks to the convexity of $g$.

Before discussing the convergence of the algorithm, let us give a proper definition of sharpness for problem \eqref{prob: composite primal}.

\begin{definition}[Sharpness of Primal Objective]
	\label{def: sharpness primal composite}
	We say that the objective function $f+g\circ L$ is sharp with respect to the set of minimizers $S_P = \argmind{}{(f+g\circ L)}\neq\emptyset$ if there exists a positive constant $\mu$ such that
	\begin{equation}
		\label{eq: def sharpness primal composite}
		(\forall x\in X) \quad (f+g\circ L) (x) - \inf_{z\in X } (f+g\circ L)(z) \geq\mu\text{dist} \left(x,S_P\right).
	\end{equation}
\end{definition}
Similarly, we can also define sharpness for the Lagrangian dual problem \eqref{prob: dual g convex, f weakly convex} as follow: there exists $\eta >0$ such that
\begin{equation}
	\label{eq: def sharpness dual Lagrange}
	(\forall y_0 \in Y) \quad \sup_{y\in Y}\inf_{x\in X} \mathcal{L}(x,y) - \inf_{x\in X} \mathcal{L}(x,y_0) \geq \eta \dist(y_0,S_D),
\end{equation}
where $S_D = \argmax_{y\in Y} \{\inf_{x\in X} \mathcal{L}(x,y)\}$ is the solution to the dual problem.

The relationship between the sharpness of the primal problem and the sharpness of Lagrangian is described in the next proposition.
\begin{proposition}
	Let us consider the composite problem \eqref{prob: composite primal} and its corresponding Lagrangian given in \eqref{eq: lagrangians}.
	If the Lagrangian is sharp in the sense of Definition \ref{def: sharpness Lagrangian nonseparable} then the primal objective function is sharp as in Definition \ref{def: sharpness primal composite}.
	On the other hand, if the objective primal $f+g\circ L$ is sharp according to Definition \ref{def: sharpness primal composite} and the respective Lagrange dual is sharp as in \eqref{eq: def sharpness dual Lagrange} then there exists $\xi>0$ such that
	\[
	\mathcal{G}_\mathcal{L} (x,y) \geq \xi (\dist(x,S_P) +\dist(y,S_D)),
	\]
	where $S_P,S_D$ are the solution sets of the primal \eqref{def: sharpness primal composite} and dual \eqref{eq: def sharpness dual Lagrange} problems, respectively. 
\end{proposition}
\begin{proof}
	By Definition \ref{def: sharpness Lagrangian nonseparable},
	we have
	\[
	\mathcal{G}_\mathcal{L} (x,y) = f(x)+g^*(y) + \sup_{(\hat{x},\hat{y}) \in X\times Y} \langle Lx,\hat{y}\rangle-\langle L\hat{x},y \rangle - f(\hat{x}) - g^*(\hat{y}) \geq \mu \dist((x,y),S),
	\]
	where $S$ is the set of saddle points.
	The RHS of the above inequality can be written
	\begin{align}
		\dist ((x,y),S) & = \inf_{(\bar{x},\bar{y})\in S} \Vert (x,y) - (\bar{x},\bar{y}) \Vert_{X\times Y} = \inf_{(\bar{x},\bar{y})\in S} \sqrt{\Vert x- \bar{x} \Vert_{X}^2 +\Vert y-\bar{y}\Vert^2_Y} \nonumber\\
		& \geq \inf_{(\bar{x},\bar{y})\in S} \Vert x- \bar{x} \Vert_{X}  \geq \dist (x,S_P).
		\label{lem: proof dist S to dist SP}
	\end{align}
	On the other hand, observe that 
	\begin{align}
		\mathcal{G}_\mathcal{L} (x,y) & = f(x)+g^*(y) + \sup_{(\hat{x},\hat{y}) \in X\times Y} \left\{ \langle Lx,\hat{y}\rangle-\langle L\hat{x},y \rangle - f(\hat{x}) - g^*(\hat{y}) \right\} \nonumber\\
		& =  f(x)+g^*(y) + \sup_{\hat{x} \in X} \left\{ -\langle L\hat{x},y \rangle - f(\hat{x})\right\} +\sup_{\hat{y}\in Y} \left\{ \langle Lx,\hat{y}\rangle - g^*(\hat{y})\right\} \nonumber\\
		& = f(x)+g^*(y) + g^{**}(Lx) +\sup_{\hat{x} \in X} \left\{ -\langle L\hat{x},y \rangle - f(\hat{x})\right\} \nonumber\\
		& = f(x) + g(Lx) +\sup_{\hat{x} \in X} \left\{ -\langle L\hat{x},y \rangle - f(\hat{x}) +g^*(y)\right\}.
		\label{lem: proof gap function}
	\end{align}
	Combining \eqref{lem: proof dist S to dist SP} and \eqref{lem: proof gap function}
	\begin{equation}
		\label{lem: proof LP sharp with Sp}
		f(x) + g(Lx) +\sup_{\hat{x} \in X} \left\{ -\langle L\hat{x},y \rangle - f(\hat{x}) +g^*(y)\right\} \geq \mu \dist(x,S_P).
	\end{equation}
	Since \eqref{lem: proof LP sharp with Sp} holds for any $(x,y)\in X\times Y$, we fix $y=y^*$ as a saddle point and obtain 
	\begin{align}
		(f+g\circ L)(x) -\inf_{\hat{x}\in X} \sup_{\hat{y} \in Y} \left\{ \langle L\hat{x},\hat{y} \rangle + f(\hat{x}) -g^*(\hat{y})\right\} & \geq (f+g\circ L)(x) -\inf_{\hat{x} \in X} \left\{ f(\hat{x})+ \langle L\hat{x},y^* \rangle -g^*(y^*)\right\}  \nonumber\\
		& \geq \mu \dist(x,S_P).\nonumber
	\end{align}
	Thanks to the convexity of $g$, the Lagrange primal has the same value as problem \eqref{prob: composite primal}. Hence, we obtain
	\[
	(f+g\circ L)(x) -\inf_{\hat{x}\in X} (f+g\circ L) (\hat{x}) \geq \mu \dist(x,S_P).
	\]
	
	For the second statement, let us assume that $f+g\circ L$ is sharp, then there exists $\mu \geq 0$ such that 
	\begin{equation*}
		(f+g\circ L)(x) - \inf_{\hat{x}\in X} f(\hat{x}) +g(L \hat{x}) \geq \mu\text{dist} (x,S_P),
	\end{equation*}
	for any $x_0 \in X$ and $S_P = \argmin f+g\circ L$. 
	We use \eqref{eq: inf gap and primal} and obtain
	\begin{equation}
		\label{eq: gap func sharp prima prob}
		\mathcal{G}_\mathcal{L} (x,y) \geq \inf_{y\in Y} \mathcal{G}_\mathcal{L} (x,y) \geq (f+g\circ L)(x) -\inf_{\overline{x} \in X} (f+g\circ L)(\overline{x}) \geq \mu\text{dist} (x,S_P).
	\end{equation}
	We recall that the dual problem has the form
	\[
	\sup_{y\in Y}\inf_{x\in X} \mathcal{L}(x,y) = \sup_{y\in Y} -f^*(-L^* y) - g^*(y),
	\]
	then sharpness for the objective of the dual problem as in \eqref{eq: def sharpness dual Lagrange} is 
	\begin{equation*}
		(\exists \eta >0)(\forall y\in Y) \quad \sup_{\hat{y}\in Y} -f^*(-L^* \hat{y}) - g^*(\hat{y}) + f^*(-L^* {y}) + g^*({y}) \geq \eta \dist(y,S_D).
	\end{equation*}
	With the same argument with the primal problem as above, we get the following
	\begin{equation}
		\label{eq: gap func sharp dual prob}
		\mathcal{G}_\mathcal{L} (x,y) \geq \eta \text{dist} (y,S_D).
	\end{equation}
	Summing \eqref{eq: gap func sharp prima prob} and \eqref{eq: gap func sharp dual prob},
	\begin{align}
		\label{eq: gap func sharp with primal-dual probs}
		2 \mathcal{G}_\mathcal{L} (x,y) \geq \mu \text{dist} (x,S_P) + \eta \text{dist} (y,S_D) & \geq \min\{\mu,\eta\} \left( \text{dist} (x,S_P) + \text{dist} (y,S_D)\right).
	\end{align}
	Denoting $\xi = \min\{\mu,\eta\}/2$, we finish the proof.
\end{proof}

The above proposition shows the relationship between the gap function and the primal problem. For the modified gap function $\mathcal{H}_\mathcal{L} (x,y)$, we obtain the following estimation.
\begin{lemma}
    Let us consider the composite problem \eqref{prob: composite primal} and its corresponding Lagrangian given in \eqref{eq: lagrangians}. Assume that $f+g\circ L$ is bounded from below, the set of minimizer $S_P$ of problem \eqref{prob: composite primal}, the set of saddle point $S$ of Lagrangian problems are nonempty. For any saddle point $(x^*,y^*)\in S$ and  $(x,y)\in X\times Y$, the following holds,
    \begin{equation}
    \label{eq: primal prob and Lagrangian}
        (f+g\circ L) (x) - \inf (f+g \circ L) \leq \mathcal{L}\left(x,y^{*}\right)-\mathcal{L}\left(x^{*},y\right)+\left\langle Lx-Lx^{*},\overline{y}-y^{*}\right\rangle,
    \end{equation}
where $\overline{y}\in\partial g(Lx)$. Moreover, if problem \eqref{prob: composite primal} is sharp in the sense of Definition \ref{def: sharpness primal composite} with $\mu >0$ and $g,g^*$ are  $\eta$ and $\eta^{-1}$-strongly convex, respectively, with $\eta>0$. Then, we have 
\[
\mathcal{H}_\mathcal{L} (x,y) \geq \mu \dist (x,S_P),
\]
for all $(x,y) \in X\times Y, y \in \partial g(Lx)$, where $\mathcal{H}(x,y)$ is the modified gap function \eqref{eq: def sharpness inf-Lagrangian}.
\end{lemma}

\begin{proof}
Let $x^{*}\in S_{P}$ then for any $y\in Y$, we have 
\begin{align*}
-\inf_{\overline{x} \in X} (f+g\circ L)(\overline{x}) = -\left(f+g\circ L\right)\left(x^{*}\right) & =-\mathcal{L}(x^*,y) -g(Lx^*)-g^*(y) + \langle Lx^*,y\rangle.
\end{align*}
On the other hand, for any $x\in X$, there exists $y^*$ corresposding to $x^*$ such that $(x^*,y^*)\in S$ is a saddle point
\begin{align*}
\left(f+g\circ L\right)\left(x\right) & =\mathcal{L}\left(x,y^{*}\right)+g\left(Lx\right)+g^{*}\left(y^{*}\right)-\left\langle Lx,y^{*}\right\rangle.
\end{align*}
Summing the two and taking $\overline{y}\in \partial g(Lx)$, we obtain
\begin{align}
    & \left(f+g\circ L\right)\left(x\right) - \inf \left(f+g\circ L\right) \nonumber\\
    & =\mathcal{L}(x,y^*) - \mathcal{L}(x^*,y) +\left[ -g(Lx^*)-g^*(y) + \langle Lx^*,y\rangle +g\left(Lx\right)+g^{*}\left(y^{*}\right)-\left\langle Lx,y^{*}\right\rangle\right]  \nonumber\\
    & \leq \mathcal{L}(x,y^*) - \mathcal{L}(x^*,y) + \left[ \langle \overline{
    y}, L(x-x^*)\rangle + \langle Lx^*,y^*-y\rangle
    +\langle Lx^*,y\rangle -\left\langle Lx,y^{*}\right\rangle
    \right] \label{eq: cv subdiff for g}\\
    & = \mathcal{L}(x,y^*) - \mathcal{L}(x^*,y) +\langle L(x-x^*), \overline{y}-y^*\rangle. \nonumber
\end{align}
This holds for any $(x,y)\in X\times Y$ and $\overline{y}\in\partial g(Lx)$. Hence, we have \eqref{eq: primal prob and Lagrangian}.

Let us assume that the primal problem \eqref{prob: composite primal} is sharp with $\mu >0$  \eqref{eq: def sharpness primal composite}, we have
\[
\mu\text{dist }\left(x,S_{P}\right)\leq\left(f+g\circ L\right)\left(x\right)-\left(f+g\circ L\right)\left(x^{*}\right)\leq\mathcal{L}\left(x,y^{*}\right)-\mathcal{L}\left(x^{*},y\right)+\left\langle Lx-Lx^{*},\overline{y}-y^{*}\right\rangle ,
\]
where $\left(x^{*},y^{*}\right)\in S$ is the saddle point. Then taking
the infimum w.r.t $\left(x^{*},y^{*}\right)\in S$, we obtain 
\[
\mu\text{dist }\left(x,S_{P}\right)+\inf_{\left(x^{*},y^{*}\right)\in S}-\left\langle Lx-Lx^{*},\overline{y}-y^{*}\right\rangle \leq\inf_{\left(x^{*},y^{*}\right)\in S}\mathcal{L}\left(x,y^{*}\right)-\mathcal{L}\left(x^{*},y\right)=\mathcal{H}\left(x,y\right)
\]
for any $\left(x,y\right)\in X\times Y$.

On the other hand, when both $g^*$ and $g$ are $\eta^{-1}$ and $\eta$-strongly convex, respectively, \cite[Proposition 14.2]{Bau2011}), from \eqref{eq: cv subdiff for g}, we can have the following form
\begin{align*}
\mu\text{dist }\left(x,S_{P}\right) & \leq\mathcal{L}\left(x,y^{*}\right)-\mathcal{L}\left(x^{*},y\right)+\left\langle Lx-Lx^{*},\overline{y}-y^{*}\right\rangle -\frac{1}{2\eta}\Vert y-y^*\Vert^2 -\frac{\eta}{2}\Vert Lx-Lx^*\Vert^2\\
& = \mathcal{L}\left(x,y^{*}\right)-\mathcal{L}\left(x^{*},y\right) -\frac{1}{2\eta}\Vert y-y^*\Vert^2 -\frac{\eta}{2}\Vert Lx-Lx^*\Vert^2 \\
& +\frac{1}{2}\left( \alpha \Vert Lx- L x^* \Vert^2 + \frac{1}{\alpha} \Vert \overline{y}-y^*\Vert^2 -\Vert \sqrt{\alpha} {Lx- L x^*} -  \frac{\overline{y}-y^*}{\sqrt{\alpha}}\Vert^2 \right) \\
& \leq \mathcal{L}\left(x,y^{*}\right)-\mathcal{L}\left(x^{*},y\right) +\frac{1}{2\alpha} \Vert \overline{y}-y^*\Vert^2 - \frac{1}{2\eta} \Vert y-y^*\Vert^2+ \frac{\alpha-\eta}{2}\Vert Lx - Lx^*\Vert^2 \\
& \leq \mathcal{L}\left(x,y^{*}\right)-\mathcal{L}\left(x^{*},y\right) +\frac{1}{2\alpha} \Vert \overline{y}-y^*\Vert^2 - \frac{1}{2\eta} \Vert y-y^*\Vert^2.
\end{align*}
where $\eta >\alpha >0$. Letting $y=\overline{y}$ and taking the infimum with respect to $(x^*,y^*) \in S$, and we finish the proof.
\end{proof}

Before using primal sharpness to discuss convergence of primal-dual algorithm, we need to transform the Lagrangian into primal objective. For $(x,y),(x_0,y_0) \in X\times Y$, we consider
\begin{align}
	\mathcal{L}(x,y_0) -\mathcal{L}(x_0,y) &= f(x)-f(x_0) +g^*(y)-g^*(y_0) +\langle Lx,y_0\rangle - \langle Lx_0,y\rangle \nonumber\\
	& \geq f(x)-f(x_0)-g^*(y_0) +\langle Lx,y_0\rangle -g^{**}(Lx_0) \nonumber\\
	& = (f+g\circ L)(x) - (f+g\circ L)(x_0) -g^*(y_0) +\langle Lx,y_0\rangle - g(Lx),
	\label{eq: Lagrange to primal obj}
\end{align}
where we utilize the convexity of $g$ in the last equality. Let us fix $(x,y)$ and take $y_0 \in \partial g(Lx)$, then for any $x_0\in X$, \eqref{eq: Lagrange to primal obj} becomes
\begin{equation}
	\mathcal{L}(x,y_0) -\mathcal{L}(x_0,y) \geq (f+g\circ L)(x) - (f+g\circ L)(x_0).
	\label{eq: Lagrange to primal obj-1}
\end{equation}
Inequality \eqref{eq: Lagrange to primal obj-1} indicates that the choice of $y_0$ depends on $x$. For primal-dual algorithms, one often arrive at the term $\mathcal{L}(x_{n+1},y_0) -\mathcal{L}(x_0,y_{n+1})$ for any $(x_0,y_0)\in X\times Y$, and so we need to know $x_{n+1}$ before choosing $y_0$. In this case, Algorithm \ref{alg: prima-dual no with primal update 1st} is more suitable for us as we perform the primal update before the dual. Therefore, we focus on Algorithm \ref{alg: prima-dual no with primal update 1st} in this section. This will be made clearer in the following result. 

\begin{proposition}
	Let $X$ and $Y$ be Hilbert, $f:X\to (-\infty,+\infty]$ be a proper lsc $\rho$-weakly convex function, $g:Y\to (-\infty,+\infty]$ be proper lsc convex and $L:X\to Y$ be a bounded linear operator. Let $(x_n)_{n\in\N},(y_n)_{n\in\N}$ be the sequences generated by Algorithm \ref{alg: prima-dual no with primal update 1st}. Let us assume that the stepsizes $\tau,\sigma$ satisfy the inequality $\sigma\rho+\sqrt{\sigma\tau}\Vert L\Vert <1$, and primal objective is sharp as in  Definition \ref{def: sharpness primal composite} with respect to the set of $S_P = \arg\min (f+g\circ L) \neq \emptyset$. Then we have the following,
	\begin{align}
		\frac{1}{2\sigma} \dist^2 (x_n,S_P) +\frac{1}{2\tau} \dist^2 (y_n, \partial g(Lx_{n+1})) & \geq \frac{1-\sigma\rho -  \sqrt{\sigma\tau} \Vert L\Vert}{2\sigma} \dist^2 (x_{n+1},S_P) +\mu \dist (x_{n+1},S_P) .
		\label{eq: f+gL sharp dist estimate}
	\end{align}
\end{proposition}
\begin{proof}
	Let us recall \eqref{eq: pd-xy norm estimation xn xn+1} in Proposition \ref{prop: pd-xy Lagrange estimation}, for any $(x,y)\in X\times Y$,
	\begin{align}
		& \mathcal{L}\left(x,y_{n+1}\right)-\mathcal{L}\left(x_{n+1},y\right)\nonumber\\
		\geq  & \frac{1- \theta\sqrt{\sigma\tau}\Vert L \Vert}{2\tau} \left\Vert y-y_{n+1}\right\Vert ^{2}-\frac{1}{2\tau}\left\Vert y-y_{n}\right\Vert ^{2} + \frac{1-  \sqrt{\sigma\tau}\Vert L \Vert}{2\tau}\left\Vert y_{n}-y_{n+1}\right\Vert ^{2}\nonumber\\
		&+\frac{1-\sigma\rho - \sqrt{\sigma\tau}\Vert L \Vert}{2\sigma} \left\Vert x-x_{n+1}\right\Vert ^{2}- \frac{1}{2\sigma}\left\Vert x-x_{n}\right\Vert ^{2}+ \frac{1-\theta \sqrt{\sigma\tau}\Vert L \Vert}{2\sigma}\left\Vert x_{n}-x_{n+1}\right\Vert ^{2}.
		\label{eq: pd-xy estimation cont-1}
	\end{align}
	
	We utilise \eqref{eq: Lagrange to primal obj-1} to obtain
	\begin{equation}
		\mathcal{L}(x_{n+1},y_0) -\mathcal{L}(x,y_{n+1}) \geq (f+g\circ L)(x_{n+1}) - (f+g\circ L)(x),
		\label{eq: primal sharpness Lagrangian}
	\end{equation}
	for any $x\in X$ and $y=y^*_{n+1} \in \partial g(Lx_{n+1})$. Combining \eqref{eq: primal sharpness Lagrangian}, sharpness condition \eqref{eq: def sharpness primal composite} and take $x^*\in S_P$, \eqref{eq: pd-xy estimation cont-1} becomes
	\begin{align*}
		-\mu\text{dist} \left(x_{n+1},S_P\right) &\geq   \frac{1- \theta\sqrt{\sigma\tau}\Vert L \Vert}{2\tau} \left\Vert y^*_{n+1}-y_{n+1}\right\Vert ^{2}-\frac{1}{2\tau}\left\Vert y^*_{n+1} -y_{n}\right\Vert ^{2} + \frac{1- \sqrt{\sigma\tau}\Vert L \Vert}{2\tau}\left\Vert y_{n}-y_{n+1}\right\Vert ^{2}\\
		&+\frac{1-\sigma\rho - \sqrt{\sigma\tau}\Vert L \Vert}{2\sigma} \left\Vert x^* -x_{n+1}\right\Vert ^{2}- \frac{1}{2\sigma}\left\Vert x^* -x_{n}\right\Vert ^{2}+ \frac{1-\theta\sqrt{\sigma\tau}\Vert L \Vert}{2\sigma}\left\Vert x_{n}-x_{n+1}\right\Vert ^{2}.
	\end{align*}
	As $y^*_{n+1} \in\partial g(Lx_{n+1})$ and $x^* \in S_P$ can be taken arbitrarily, we obtain
	\begin{align*}
		-\mu\text{dist} \left(x_{n+1},S_P\right) &\geq   \frac{1- \theta\sqrt{\sigma\tau}\Vert L \Vert}{2\tau} \dist^2 (y_{n+1}, \partial g(Lx_{n+1})) -\frac{1}{2\tau} \dist^2 (y_n, \partial g(Lx_{n+1}) ) \\
		&+\frac{1-\sigma\rho - \sqrt{\sigma\tau}\Vert L \Vert}{2\sigma} \dist^2 (x_{n+1},S_P) - \frac{1}{2\sigma} \dist^2 (x_n,S_P)\\
        & + \frac{1- \sqrt{\sigma\tau}\Vert L \Vert}{2\tau}\left\Vert y_{n}-y_{n+1}\right\Vert ^{2} + \frac{1-\theta\sqrt{\sigma\tau}\Vert L \Vert}{2\sigma}\left\Vert x_{n}-x_{n+1}\right\Vert ^{2}.
	\end{align*}
	By assumption on the stepsizes $\tau,\sigma$, we obtain \eqref{eq: f+gL sharp dist estimate}.
\end{proof}

This shows that the convergence of $(\text{dist} \left(x_{n},S_P\right) )_{n\in\N}$ depends on the distance of the dual sequence $(y_n)_{n\in\N}$ to $\partial g(Lx_{n+1})$. This occurs since we only make assumption on the primal objective, and do not have additional information on the dual variable $y\in Y$ like in Definition \ref{def: sharpness inf-Lagrangian}. 
Hence, one can think of the term $(\text{dist} \left(y_{n},\partial g(Lx_{n+1}\right) )_{n\in\N}$ as the inexactness produce by the algorithm, despite the fact that proximal calculation is exact. For simplicity, we denote $\varepsilon_n := \dist (y_n,\partial g(Lx_{n+1})) \geq 0$ for all $n\in \N$. This quantity can be checked after the primal update and with $\varepsilon_n$ small enough, we can obtain convergence, as we will see below.

\begin{lemma}
	\label{lem: inside tube with eps_n}
	Let $X$ and $Y$ be Hilbert, $f:X\to (-\infty,+\infty]$ be a proper lsc $\rho$-weakly convex function, $g:Y\to (-\infty,+\infty]$ be proper lsc convex and $L:X\to Y$ be a bounded linear operator. Let $(x_n)_{n\in\N},(y_n)_{n\in\N}$ be the sequences generated by Algorithm \ref{alg: prima-dual no with primal update 1st}. Let us assume that the stepsizes $\tau,\sigma$ satisfy the inequality $\sigma\rho+\sqrt{\sigma\tau}\Vert L\Vert <1$, where $\theta\in[0,1]$, and primal objective is sharp as in  Definition \ref{def: sharpness primal composite} with respect to the set of $S_P = \arg\min (f+g\circ L) \neq \emptyset$. If $\sigma,\tau$ satisfy the following
	\[
	(\forall n\in \N) \quad \mu^2 \sigma\tau > (\sigma \rho + \sqrt{\sigma \tau} \Vert L\Vert) \varepsilon_n^2, \ \varepsilon_n = \dist (y_n,\partial g(Lx_{n+1})),
	\]
	and at iteration $n\in\N$,
	\begin{align}
		\label{eq: En tube}
		\text{dist}\left(x_{n},S_P\right)\leq
		E^+_n \quad & \Longrightarrow \quad \dist(x_{n+1},S_P) \leq E^+_n, \\
		\text{dist}\left(x_{n},S_P\right)\leq
		E^-_n \quad & \Longrightarrow \quad \dist(x_{n+1},S_P) \leq E^-_n,
		\label{eq: En minus tube}
	\end{align}
	where 
	\[
	E_n^+ := \frac{\sigma\mu +\sigma \sqrt{\mu^2 - \frac{\varepsilon_n^2 (\sigma\rho +\sqrt{\sigma\tau} \Vert L\Vert)}{\sigma\tau}}}{\sigma\rho +\sqrt{\sigma\tau} \Vert L\Vert}, \quad 
	E_n^- := \frac{\sigma\mu -\sigma \sqrt{\mu^2 - \frac{\varepsilon_n^2 (\sigma\rho +\sqrt{\sigma\tau} \Vert L\Vert)}{\sigma\tau}}}{\sigma\rho +\sqrt{\sigma\tau} \Vert L\Vert}.
	\]
	Moreover, if 
	\begin{equation}
		\label{eq: En < dist <En}
		E_n^- \leq \dist (x_{n+1},S_P) \leq E_n^+,
	\end{equation}
	Then $\dist(x_{n+1},S_P) \leq \dist(x_n,S_P)$.
\end{lemma}
\begin{proof}
	From \eqref{eq: f+gL sharp dist estimate}, we have
	\begin{equation}
		\label{eq: f+gL sharp dist estimate eps}
		\frac{1}{2\sigma} \dist^2 (x_n,S_P) \geq \frac{1-\sigma\rho - \sqrt{\sigma\tau} \Vert L\Vert}{2\sigma} \dist^2 (x_{n+1},S_P) +\mu \dist (x_{n+1},S_P) - \frac{\varepsilon_n^2}{2\tau}.
	\end{equation}
	We would like the RHS of \eqref{eq: f+gL sharp dist estimate eps} satisfies
	\begin{equation*}
		\frac{1-\sigma\rho -  \sqrt{\sigma\tau} \Vert L\Vert}{2\sigma} \dist^2 (x_{n+1},S_P) +\mu \dist (x_{n+1},S_P) - \frac{\varepsilon_n^2}{2\tau}
		\geq \frac{1}{2\sigma} \dist^2 (x_{n+1},S_P),
	\end{equation*}
	which is equivalent to solving the following quadratic inequality
	\[
	\frac{\sigma\rho + \sqrt{\sigma\tau} \Vert L\Vert}{2\sigma} z^2 -\mu z + \frac{\varepsilon_n^2}{2\tau} \leq 0.
	\]
	Thanks to the assumption $ \mu^2 \sigma\tau > (\sigma \rho + \sqrt{\sigma \tau} \Vert L\Vert) \varepsilon_n^2$, we obtain the solution
	\[
	\frac{\sigma\mu -\sigma \sqrt{\mu^2 - \frac{\varepsilon_n^2 (\sigma\rho +\sqrt{\sigma\tau} \Vert L\Vert)}{\sigma\tau}}}{\sigma\rho +\sqrt{\sigma\tau} \Vert L\Vert} 
	\leq \dist (x_{n+1},S_P) \leq 
	\frac{\sigma\mu +\sigma \sqrt{\mu^2 - \frac{\varepsilon_n^2 (\sigma\rho +\sqrt{\sigma\tau} \Vert L\Vert)}{\sigma\tau}}}{\sigma\rho +\sqrt{\sigma\tau} \Vert L\Vert},
	\]
	which is \eqref{eq: En < dist <En}. If \eqref{eq: En < dist <En} holds, we obtain $\dist(x_{n+1},S_P) \leq \dist (x_n,S_P)$.
	
	Let us prove \eqref{eq: En tube}, assume that for $n\in \N$, we have $	\dist (x_n,S_P) \leq E_n^+$.	By \eqref{eq: f+gL sharp dist estimate eps}, we have
	\begin{equation}
		\label{eq: bound on dxn+1}
		\frac{1}{2\sigma} (E_n^+)^2 \geq \frac{1-\sigma\rho -  \sqrt{\sigma\tau} \Vert L\Vert}{2\sigma} \dist^2 (x_{n+1},S_P) +\mu \dist (x_{n+1},S_P) - \frac{\varepsilon_n^2}{2\tau}.
	\end{equation}
	By solving \eqref{eq: bound on dxn+1}, we obtain,
	$
	0\leq d(x_{n+1},S_P) \leq E_n^+.
	$
	The proof for \eqref{eq: En minus tube} can be done in the same way.
\end{proof}

Lemma \ref{lem: inside tube with eps_n} describes the behavior between $\dist(x_{n+1},S_P)$ and $E_n^+$. Condition \eqref{eq: En < dist <En} is quite restrictive as we need to bound $\dist(x_{n+1},S_P)$ on both sides. In the next result ,we give a relaxed estimation of $(\dist(x_{n},S_P))_{n\in\N}$.

\begin{theorem}
	\label{thm: convergence primal sharp}
	Let $X$ and $Y$ be Hilbert, $f:X\to (-\infty,+\infty]$ be a proper lsc $\rho$-weakly convex function, $g:Y\to (-\infty,+\infty]$ be proper lsc convex and $L:X\to Y$ be a bounded linear operator. Let $(x_n)_{n\in\N},(y_n)_{n\in\N}$ be the sequences generated by Algorithm \ref{alg: prima-dual no with primal update 1st}. Let us assume that the stepsizes $\tau,\sigma$ satisfy the inequality $\sigma\rho+\sqrt{\sigma\tau}\Vert L\Vert <1$, and primal objective is sharp as in  Definition \ref{def: sharpness primal composite} with respect to the set of $S_P = \arg\min (f+g\circ L) \neq \emptyset$. If $\sigma,\tau$ satisfy the following
	\[
	(\forall n\in \N) \quad \mu^2 \sigma\tau > (\sigma \rho + \sqrt{\sigma \tau} \Vert L\Vert) \varepsilon_n^2,
	\]
	and at iteration $n\in\N$,
$	
 \text{dist}\left(x_{n},S_P\right)\leq
		E^+_n.
$
	Then there exists $\xi_n \geq 1$ such that
	\begin{equation}
		\label{eq: estimation dist - En}
		\xi_{n+1} \left[ \dist^2 (x_{n+1},S_P) - (E_n^-)^2\right] \leq \dist^2 (x_n,S_P) - (E_n^-)^2.
	\end{equation}
	On the other hand, if $E_0^- \leq \dist (x_0, S_P) \leq E_0^+$ and $\varepsilon_0 >0, \sum \varepsilon_n^2 <+\infty$ then $\dist (x_{n},S_P) \to 0$ as $n\to \infty$.
\end{theorem}
\begin{proof}
	From the proof of Lemma \ref{lem: inside tube with eps_n}, we know that $E_n^+, E_n^-$ are solutions of the following equation
	\[
	\frac{\sigma \rho + \sqrt{\sigma \tau} \Vert L\Vert}{2\sigma} z^2 -\mu z + \frac{\varepsilon_n^2}{2\tau} = 0. 
	\]
	We subtract $(E_n^-)^2 / 2\sigma$ on both sides of \eqref{eq: f+gL sharp dist estimate},
	\begin{align*}
		\frac{1}{2\sigma} \left[\dist^2 (x_n,S_P) -(E_n^-)^2 \right] & \geq \frac{1-\sigma\rho -  \sqrt{\sigma\tau} \Vert L\Vert}{2\sigma} \left[\dist^2 (x_{n+1},S_P) -(E_n^-)^2\right] +\mu (\dist (x_{n+1},S_P) - E_n^-) \nonumber \\
		& = \xi_{n+1} \left[\dist^2 (x_{n+1},S_P) -(E_n^-)^2\right],    
	\end{align*}
	where 
	\[
	\xi_{n+1} := \frac{1-\sigma\rho -  \sqrt{\sigma\tau} \Vert L\Vert}{2\sigma} +\frac{\mu}{\dist (x_{n+1},S_P) + E_n^-} > 0.
	\]
	Notice that $\xi_{n+1} \geq 1/2\sigma$ if $\dist(x_{n+1},S_P) \leq E_n^+$ which is satisfied by the result in Lemma \ref{lem: inside tube with eps_n}.
	
	For the second statement, by assumption
	\[ 
	E_0^- \leq \dist (x_0,S_P) \leq E_0^+ < E^+ := \frac{2\sigma \mu}{\sigma\rho +  \sqrt{\sigma\tau} \Vert L\Vert}. 
	\]
	Hence, from \eqref{eq: f+gL sharp dist estimate}, we have
	\begin{equation}
		\label{eq: sum form of dist}
		\frac{1}{2\sigma} \left[\dist^2 (x_n,S_P) - \dist^2 (x_{n+1},S_P)\right] + \frac{\varepsilon_n^2}{2\tau} 
		\geq \dist(x_{n+1},S_P) \left[ \mu -\frac{\sigma\rho + \sqrt{\sigma\tau} \Vert L\Vert}{2\sigma} \dist (x_{n+1},S_P) \right] .
	\end{equation}
	The RHS of \eqref{eq: sum form of dist} is nonnegative thanks to assumption on $\dist(x_0,S_P)$. Indeed, from Lemma \ref{lem: inside tube with eps_n},  $\dist (x_1,S_P) \leq E_0^+$. If $E_0^- \leq\dist (x_0,S_P) \leq \dist (x_1,S_P) \leq E_0^+$  then $\dist (x_0,S_P) = \dist (x_1,S_P)$, otherwise $\dist(x_1,S_P) \leq E_0^- < E^+$. By induction,  $\dist(x_n,S_P)$ is is bounded away from $E^+$ for all $n\in\N$.
	Taking the sum of \eqref{eq: sum form of dist}, we obtain
	\[
	+\infty >\frac{1}{2\sigma} \dist^2 (x_0,S_P) + \sum_{n=0}^{\infty} \frac{\varepsilon_n^2}{2\tau} 
	\geq \sum_{n=0}^{\infty} \dist(x_{n+1},S_P) \left[ \mu -\frac{\sigma\rho +  \sqrt{\sigma\tau} \Vert L\Vert}{2\sigma} \dist (x_{n+1},S_P) \right],
	\]
	which implies $\dist(x_{n+1},S_P) \left[ \mu -\frac{\sigma\rho +  \sqrt{\sigma\tau} \Vert L\Vert}{2\sigma} \dist (x_{n+1},S_P) \right] \to 0$ as $n\to\infty$. Since $\dist(x_n,S_P)$ is bounded away from $E^+$, $\dist(x_n,S_P)$ musts tend to zero.
\end{proof}

This result shows that if $\dist(y_n,\partial g(Lx_{n+1})) \to 0$ fast enough then $\dist (x_n,S_P) \to 0$ as well.
If one know $\partial g(Lx)$ then we can choose $\varepsilon_n$ small enough so $\dist(x_n,S_P)$ converges.
The assumption on the stepsizes can be realized, for example: assume that $\max \varepsilon_n^2 =0.1, \rho =2, \mu^2 =1, \Vert L\Vert = 1$. Choosing $\tau = 0.5$ and $0.0555< \sigma < 0.304806$ will satisfied the condition in Theorem \ref{thm: convergence primal sharp}.

\section{Numerical Examples}
\label{sec: lsc-pd numerical}

\subsection{Inf-sharpness}
We demonstrate the performance of Algorithm \ref{alg: prima-dual no xn-1} with several examples which satisfies inf-sharpness condition in Definition \ref{def: sharpness inf-Lagrangian}. 
\begin{example}
	\label{ex: num Lagrange inf-sharp 3}
	We recall Lagrangian in Example \ref{ex: Lagrange inf-sharp 3}
	\[
	\mathcal{L} (x,y) = |x| +xy -|y|.
	\]
	The respective primal and dual problems for this Lagrangian are
	\begin{align*}
		\inf_{-1 \leq x \leq 1}  |x|, \quad
		\sup_{ -1 \leq y \leq 1} -|y|.
	\end{align*}
	\textbf{Setting:} We set $\tau =0.25, \sigma = 0.75, \theta =1,$ maximum number of iteration $N=2001$. We randomized the starting point in the interval $[-10,10]$. For each step, we calculate the proximal operator of the absolute function $f(x) =|x|$ which has an explicit form,
	\[
	\mathrm{prox}_{\gamma f}(x_0) =\begin{cases}
		x_0 - \gamma & x_0 \geq \gamma, \\
		x_0 + \gamma & x_0 \leq -\gamma,\\
		0 & -\gamma < x_0 < \gamma.
	\end{cases}
	\]
	We plot the distance at each iteration to the saddle point and the primal, dual iterate in Figure \ref{fig: ex inf-sharp Lagrange 3}.
	\begin{figure}
		\centering
		\begin{tabular}{cc}
			\makecell{\includegraphics[width=0.45\textwidth, trim ={15  0 0 0}, clip]{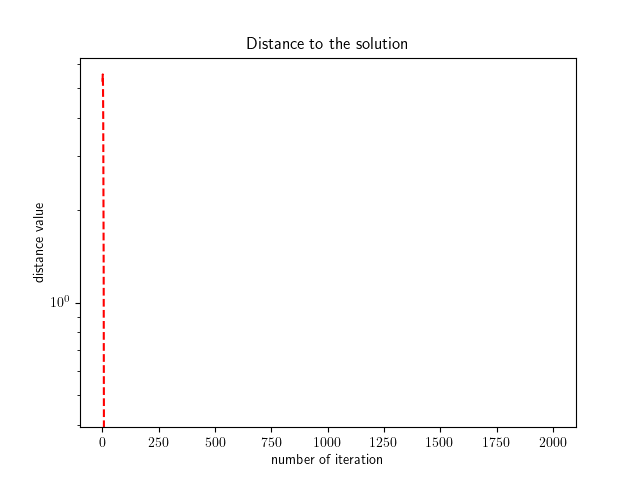}} & 
			\makecell{\includegraphics[width=0.45\textwidth, trim ={15  0 0 0}, clip]{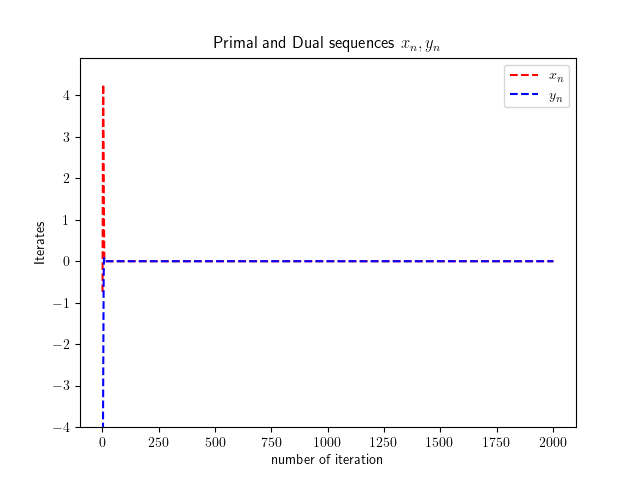}} \\
			\multicolumn{1}{c}{(a)}&  \multicolumn{1}{c}{(b)}
		\end{tabular}
		\caption{Example \ref{ex: num Lagrange inf-sharp 3}. From left to right: distance to the solution; primal-dual iterates.}
		\label{fig: ex inf-sharp Lagrange 3}
	\end{figure}
	This is a special case because by the nature of the proximal operator for $f$, which returns zero whenever the iterate is closed enough. That is why primal and dual iterate take zero value just after a few iterations.
\end{example}
{
	
	\begin{example}
		\label{ex: inf-sharp Lagrange 4}
		We consider the Lagrangian
		\[
		\mathcal{L} (x,y) = |x|+|x^2-2|+xy -|y|-|y^2-2|,
		\]
		which has a similar form as in Example \ref{ex: lagrange inf-sharp 4}. It also has one saddle point at $(0,0)$. We run Algorithm \ref{alg: prima-dual no xn-1} for this Lagrangian in two cases: inside and outside the region mentioned in Theorem \ref{theo: pd-yx seq convergence}.
		To compute the primal and dual update of Algorithm \ref{alg: prima-dual no xn-1}, we use Scipy package in Python to approximate the proximal operator in each step. 
		
		\paragraph{Setting: } We set $\tau = 0.25,\sigma =0.35, \theta=1$ and we take sharpness constant $\mu=0.9$. The weakly convex modulus $\rho=2$, so the quantity  
			$\frac{\mu }{\max\left\{ \frac{1}{2\sigma} , \frac{1}{2\tau}\right\} - A} \approx 0.3199,$ 
		where $A = \min\left\{ \frac{1- \sqrt{\sigma\tau}\Vert L \Vert}{2\tau}, \frac{1-\sigma\rho - \theta\sqrt{\sigma\tau}\Vert L \Vert}{2\sigma} \right\} \approx 0.00599.$
		For the first case, we randomize the starting points $(x_0,y_0) \in [-0.3199, 0.3199]^2$ and for the second case, $[-10,10]^2$.
		We plot the distance of the iterate to the saddle point, both primal and dual iterate in Figure \ref{fig: ex inf-sharp Lagrange 4 intube} for the first case, and Figure \ref{fig: ex inf-sharp Lagrange 4 outtube} for the second case. Notice the differences in the two figures as the sequences converge to different points from the saddle point in the second case.
		
		\begin{figure}
			\centering
			\begin{tabular}{cc}
				\makecell{\includegraphics[width=0.45\textwidth, trim ={15  0 0 0}, clip]{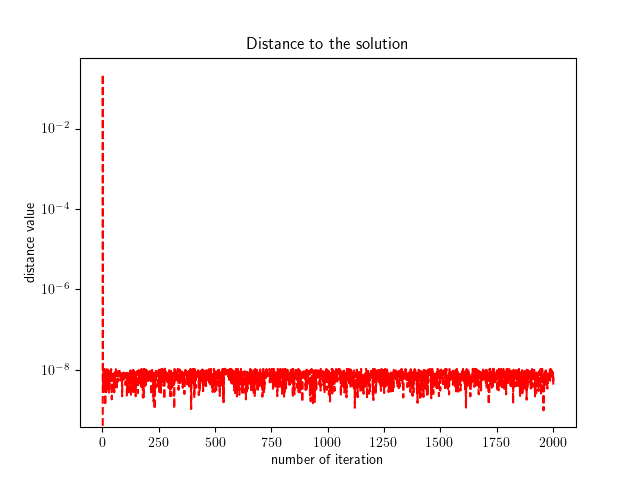}} & 
				\makecell{\includegraphics[width=0.45\textwidth, trim ={15  0 0 0}, clip]{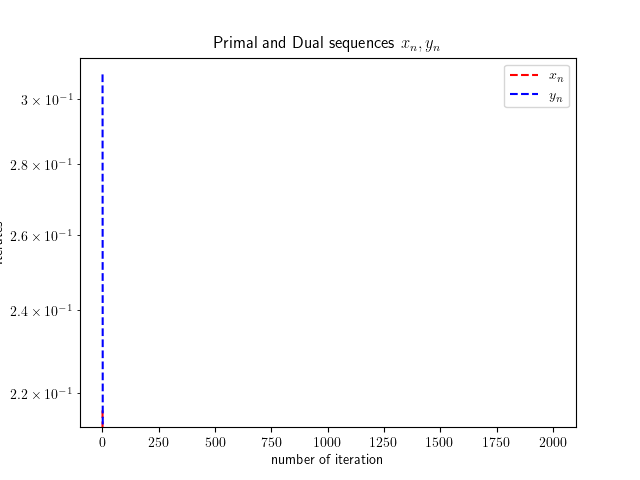}} \\
				\multicolumn{1}{c}{(a)}&  \multicolumn{1}{c}{(b)}
			\end{tabular}
			\caption{Example \ref{ex: inf-sharp Lagrange 4}; Case $1/2$; From left to right: distance to the solution; primal-dual iterates.}
			\label{fig: ex inf-sharp Lagrange 4 intube}
		\end{figure}
		\phantom{spacing for better view}
		\begin{figure}
			\centering
			\begin{tabular}{cc}
				\makecell{\includegraphics[width=0.45\textwidth, trim ={15  0 0 0}, clip]{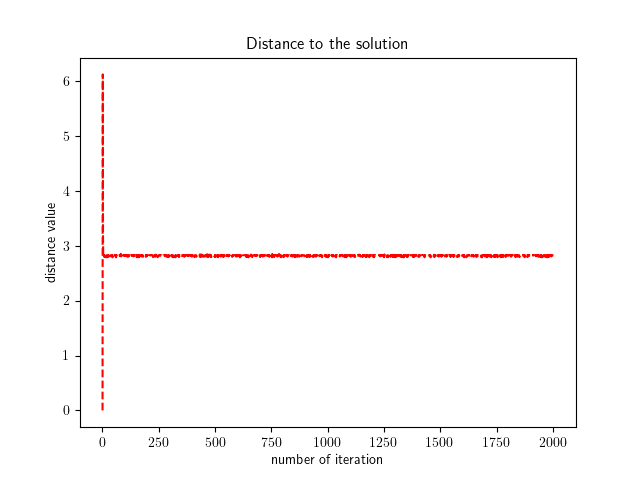}} & 
				\makecell{\includegraphics[width=0.45\textwidth, trim ={15  0 0 0}, clip]{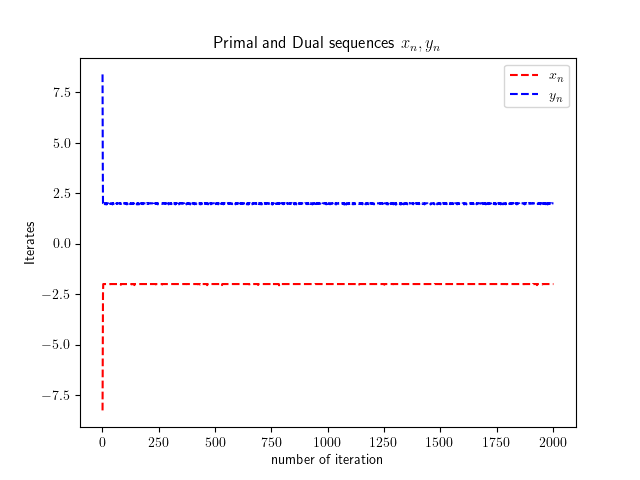}} \\
				\multicolumn{1}{c}{(a)}&  \multicolumn{1}{c}{(b)}
			\end{tabular}
			\caption{Example \ref{ex: inf-sharp Lagrange 4}; Case $2/2$; From left to right: distance to the solution; primal-dual iterates.}
			\label{fig: ex inf-sharp Lagrange 4 outtube}
		\end{figure}
	\end{example}
}
\subsection{$\ell_1$-regularization} 
\subsubsection*{Large scale $\ell_1$-regularization}
The inf-sharp condition is quite restrictive, as we need to test for Lagrangian. It is much easier to verify sharpness condition \eqref{eq: def sharpness primal composite} for primal objective functions.
For example, consider the following $\ell_1$ regularization problem in $\mathbb{R}^n$
\begin{equation}
	\label{eq: l1 problem}
	\min_{x\in \mathbb{R}^n} \Vert x \Vert_1 +\frac{1}{2}\Vert Ax-b\Vert^2,    
\end{equation}
where $\Vert x\Vert_1 = \sum_{i=1}^n |x_i|$, $A \in \mathbb{R}^{m\times n}$ is a random transformation matrix and $b\in \mathbb{R}^m$ is the noisy data.
Our job is to recover the original signal $x^*$. Problem \eqref{eq: l1 problem} is a common reconstruction problem and has been studied extensively in \nolinebreak\citep{combettes2005signal,chambolle2011first,molinari2021iterative}. 

Instead of solving problem \eqref{eq: l1 problem}, we modify it into the following
\begin{equation}
	\label{eq: l1 problem weakconv}
	\min_{x\in \mathbb{R}^n} |\Vert x\Vert^2 -\Vert x^* \Vert^2| + \frac{1}{2}\Vert Ax-b\Vert^2,
\end{equation}
which is weakly convex and sharp thanks to the first term (see \nolinebreak\cite{bednarczuk2023convergence}). Problem \eqref{eq: l1 problem weakconv} has the same meaning as problem \eqref{eq: l1 problem} as a reconstruction problem from noisy data where one does not know the original signal but its magnitude. We plan to solve both problems \eqref{eq: l1 problem} and \eqref{eq: l1 problem weakconv} using Algorithm \ref{alg: prima-dual no with primal update 1st} to compare the behavior of the sequence $(x_n)_{n\in\mathbb{N}}$. 

For proximal calculation, the function $\Vert y - b\Vert^2 /2$ is convex and has an explicit form for each update
\[
y_{n+1} = \frac{y_n - \tau b}{1+\tau}.
\]
The function $\| \Vert x\Vert^2 - c|, c>0$ is just a generalization of the one in \nolinebreak\cite[Remark 8]{bednarczuk2023convergence} so we can have an explicit update with stepsize $2\sigma <1$ as
\[
x_{n+1} = \begin{cases}
	\frac{x_{n}}{1+2\sigma} & \left\Vert x_{n}\right\Vert ^{2}>\left(1+2\sigma\right)^{2}c,\\
	\frac{\sqrt{c}x_{n}}{\left\Vert x_{n}\right\Vert } & \left(1-2\sigma\right)^{2}c\leq\left\Vert x_{n}\right\Vert ^{2}\leq\left(1+2\sigma\right)^{2}c,\\
	\frac{x_{n}}{1-2\sigma} & \left\Vert x_{n}\right\Vert ^{2}<\left(1-2\sigma\right)^{2}c.
\end{cases}
\]

\paragraph{Setting:} We intend to solve a large-scale sparse problem. Let $n=3000,m=2000$, and $A$ be a random matrix with normal distribution. For $x^*\in \mathbb{R}^n$, we randomize as a sparse vector with the density of $10\%$ with uniform distribution in $[0,1)$. For noisy data, we set $b = Ax^* +\delta$ where we consider two cases for $\delta$: $\delta = 0.1$ and $\delta = \Vert Ax^*\Vert U$ where $U$ is a random vector with uniform distribution in $[-0.1,0.1)$.

Let us set the maximum iterations be $N=5000$ and stepsizes $\sigma = 0.1,\tau = \min\{0.99, \frac{1}{\Vert A\Vert^2 \sigma}\}, \theta =1$ to ensure that $\Vert A\Vert \sqrt{\sigma\tau}<1$.
As the starting iterates can affect the algorithm, we test for two cases when $(x_0,y_0)$ are zeros vectors and when $x_0= E_0^+ x^* / \Vert x^*\Vert$ is close to the solution by the quantity
\[
E^+_0 = \frac{\sigma\mu +\sigma \sqrt{\mu^2 - \frac{\varepsilon_0^2 (\sigma\rho +\sqrt{\sigma\tau} \Vert A\Vert)}{\sigma\tau}}}{\sigma\rho +\sqrt{\sigma\tau} \Vert A\Vert}, 
\]
where $\mu =0.99,\rho=2,\varepsilon_0^2 = 10^{-7}$. This term can be considered as the bounds on the distance in Theorem \ref{thm: convergence primal sharp}.
We measure the convergence by calculating $\Vert x_n - x^*\Vert$, and illustrate the performance of Algorithm \ref{alg: prima-dual no with primal update 1st} of two problems \eqref{eq: l1 problem} and \eqref{eq: l1 problem weakconv} with the same starting points in Figures \ref{fig: ex primal sharp 01 noise} and \ref{fig: ex primal sharp uniform noise}.
\begin{figure}
	\centering
	\begin{tabular}{cc}
		\makecell{\includegraphics[width=0.45\textwidth, trim ={0 0 0 0}, clip]{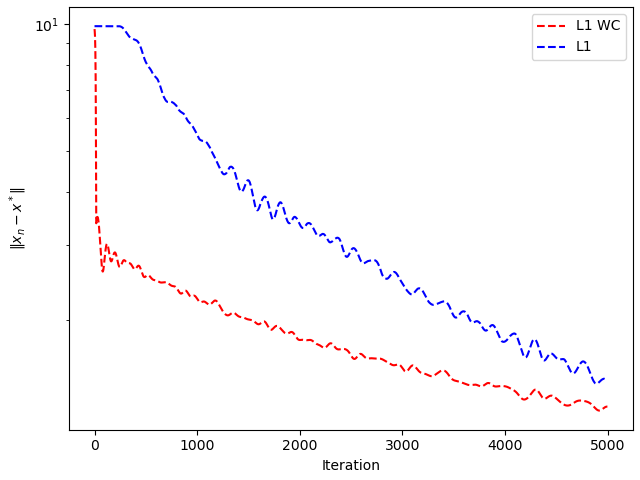}} & 
		\makecell{\includegraphics[width=0.45\textwidth, trim ={18  0 0 0}, clip]{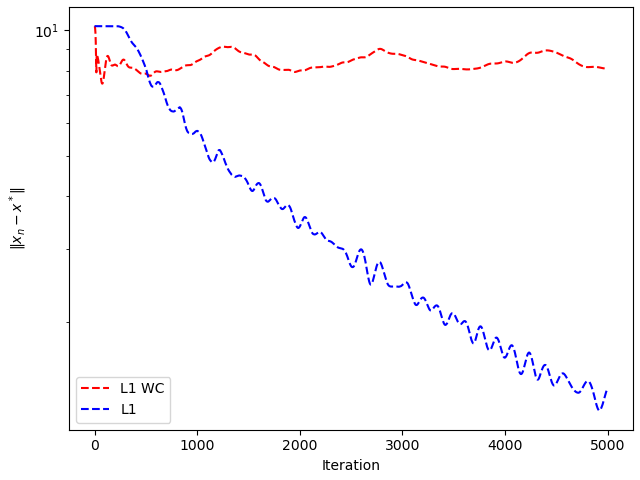}} \\
		\multicolumn{1}{c}{(a)}&  \multicolumn{1}{c}{(b)}
	\end{tabular}
	\caption{Distance to the solution of problem \eqref{eq: l1 problem} (blue) and \eqref{eq: l1 problem weakconv} (red) with initials: closed to the solution (a), zero (b).}
	\label{fig: ex primal sharp 01 noise}
\end{figure}
\phantom{aaaaa}
\begin{figure}
	\centering
	\begin{tabular}{cc}
		\makecell{\includegraphics[width=0.45\textwidth, trim ={0 0 0 0}, clip]{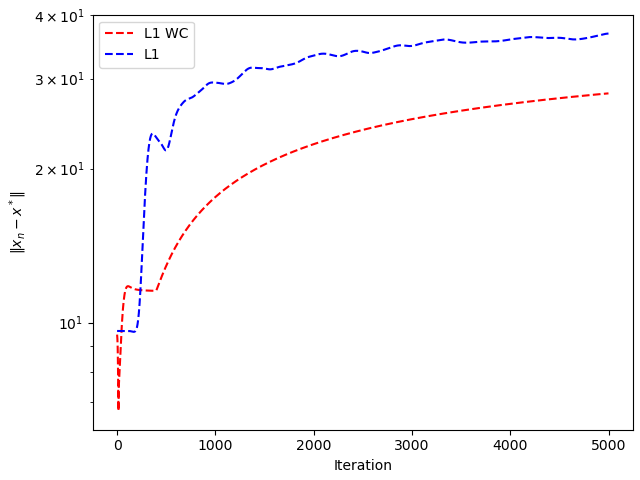}} & 
		\makecell{\includegraphics[width=0.45\textwidth, trim ={18  0 0 0}, clip]{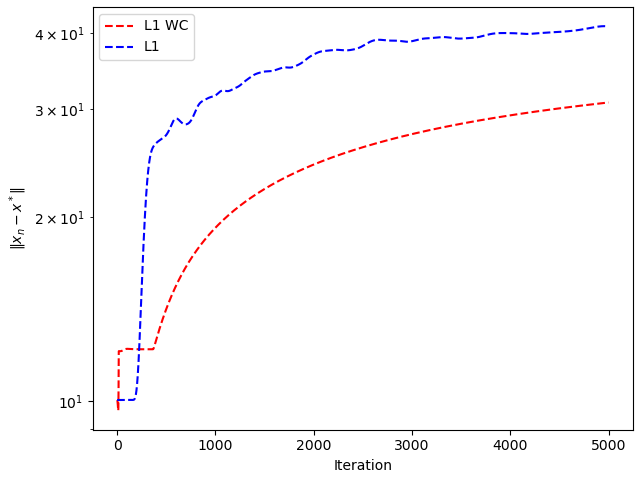}} \\
		\multicolumn{1}{c}{(a)}&  \multicolumn{1}{c}{(b)}
	\end{tabular}
	\caption{Distance to the solution of problem \eqref{eq: l1 problem} (blue) and \eqref{eq: l1 problem weakconv} (red) for large random noise with initials: closed to solution (a), zero (b).}
	\label{fig: ex primal sharp uniform noise}
\end{figure}

For the first case with constant noise, Algorithm \ref{alg: prima-dual no with primal update 1st} with close initialization to the solution gives us a better result for weakly convex problem \eqref{eq: l1 problem weakconv} than \eqref{eq: l1 problem}. For zero initialization, $x_n$ from \eqref{eq: l1 problem weakconv} tend to be bounded from below by some threshold, while the $x_n$ from \eqref{eq: l1 problem} continue toward the solution. After a particular iteration, Algorithm \ref{alg: prima-dual no xn-1} for convex problem has better results compared to weakly convex problem \eqref{eq: l1 problem weakconv}, and continues to decrease to the solution while weakly convex problem stabilizes at some value. We believe that this behavior comes from two factors: we do not put any control on the dual iterate $y_n$, and using a proximal subgradient with fixed constant $\rho$ for weakly convex function.

For the second case with larger noise from $\Vert Ax^*\Vert$, both problems diverge from the solution. As seen in Figure \ref{fig: ex primal sharp uniform noise}, in both initializations, the iterate from weakly convex problem \eqref{eq: l1 problem weakconv} has better error compared to convex problem \eqref{eq: l1 problem}.

\subsubsection*{Image Deblurring}
We use a similar model and comparison in the previous subsection to Image Deblurring.
\begin{align}
	& \min_{x} \frac{1}{2}\Vert Ax-b\Vert^2 + \Vert x\Vert_1, \tag{Convex model} \\
	& \min_{x} | \Vert x\Vert^2 - \Vert x_0\Vert^2 | + \frac{1}{2}\Vert Ax-b\Vert^2. \tag{Weakly Convex model}
\end{align}
Here $x_0$ is the original image, the matrix $A$ is a Gaussian blurring matrix, we extract $A$ by applying Gaussian filter with standard deviation $\sigma=4$ to the identity matrix. The quantity $b$ will be the blurred image corrupted with additive noise from normal distribution, $b = Ax+\varepsilon \mathcal{N}(0,1)$ where $\varepsilon = 0.01$.

For the setting of the algorithm, we keep the same stepsizes as in the previous example with zeros initialization. We only run the algorithm for $N=1000$ steps. To measure the result, we calculate the peak signal-to-noise ratio (PSNR) with respect to the original image. The higher PSNR implies a better result from the process.

The images are taken from the dataset BSD68 of \href{https://github.com/majedelhelou/denoising_datasets/tree/main}{Github}. The images are then rescaled to the value between $[0,1]$ and resized into $256\times 256$. The details for preprocessing steps can be found in \nolinebreak\cite{peyre2011numerical} or (\href{https://www.numerical-tours.com/}{Numerical Tours}). We test for several images (number $007,011,013,044$ in the dataset) and show the results in Figure \eqref{fig: Deblurring}. The result shows that the convex model for this application does not work well. The weakly convex model gives improved results compared to the blurred images, as one can recognize the shapes and objects. This is because, in the weakly convex model, we also include the information of the solution $x_0$ and the noisy data.

\begin{figure}
	\centering
	\includegraphics[width=1\textwidth, clip]{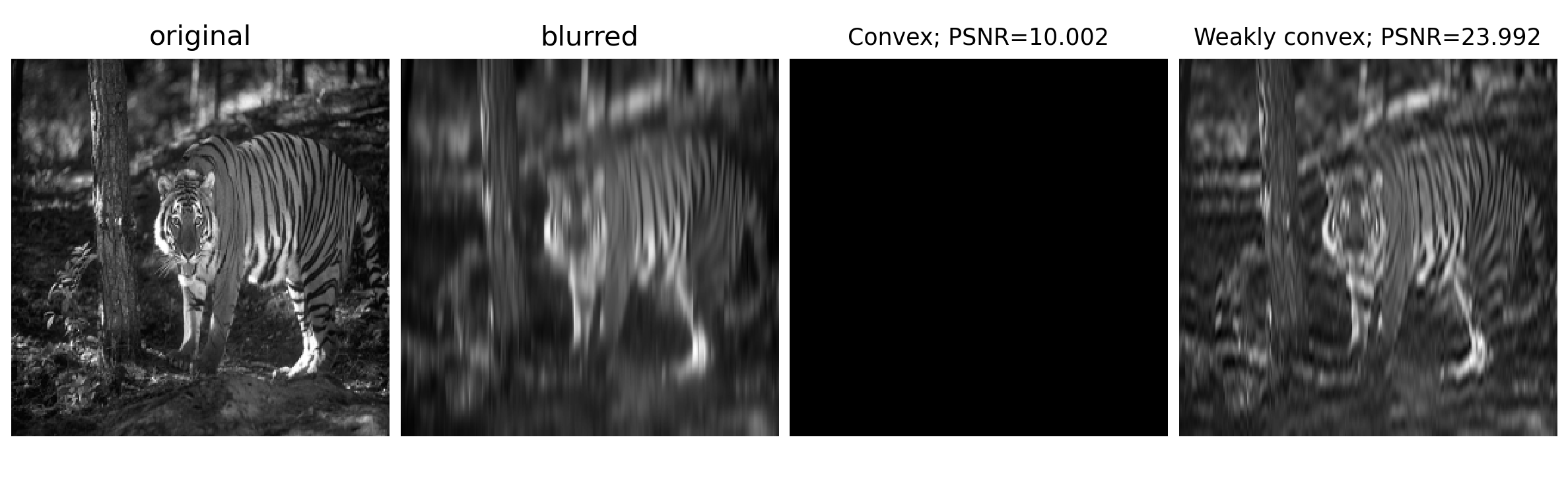} \\
	\includegraphics[width=1\textwidth, clip]{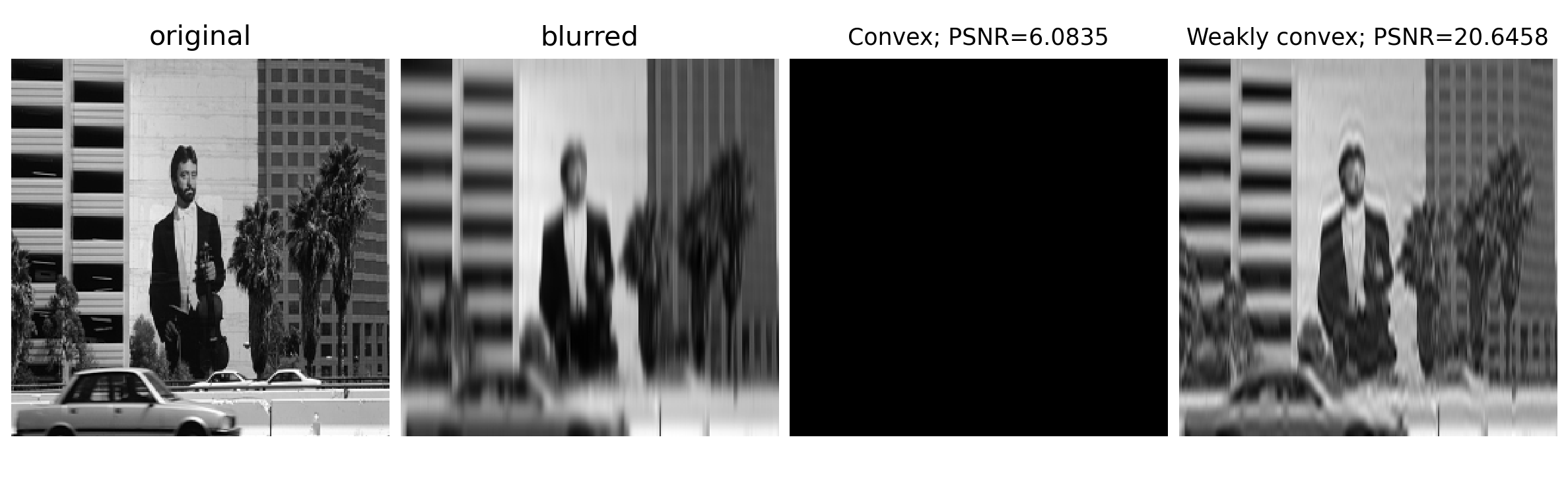} \\\includegraphics[width=1\textwidth, clip]{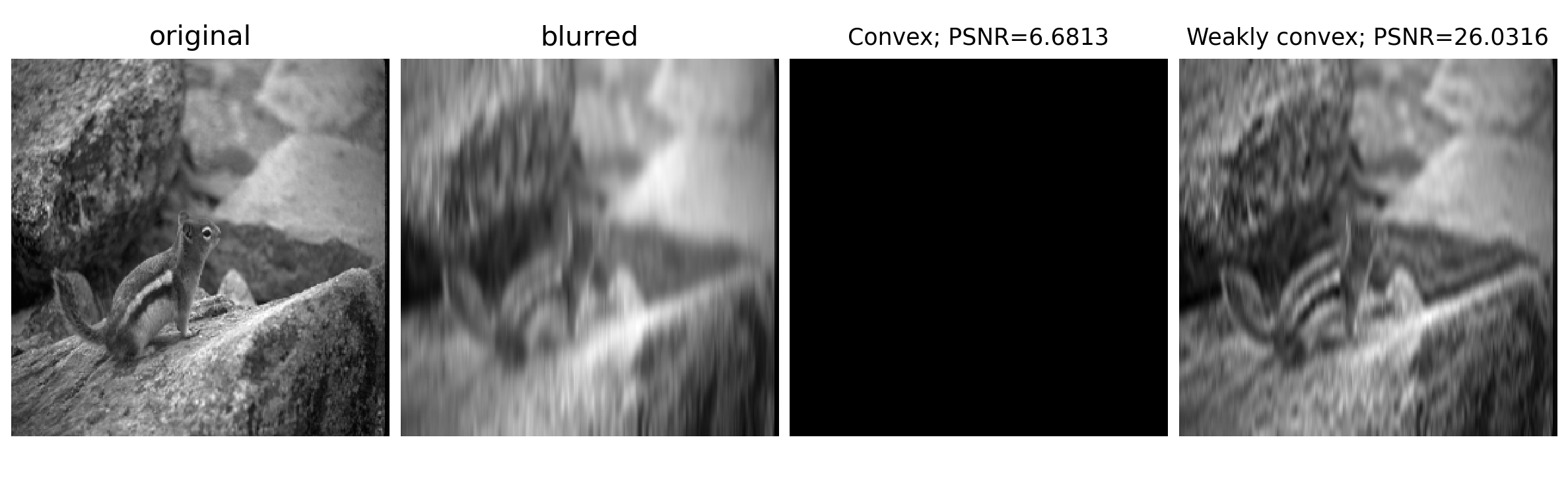} \\
	\includegraphics[width=1\textwidth, clip]{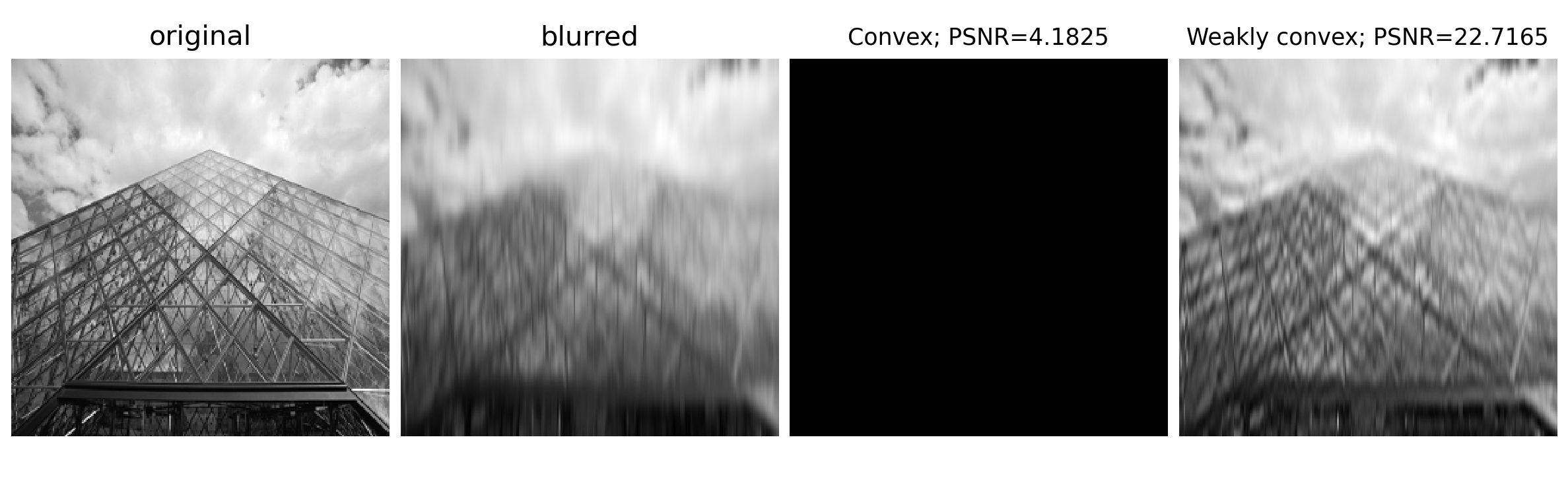} 
	\caption{Image Deblurring using convex and weakly convex model.}
	\label{fig: Deblurring}
\end{figure}

\subsection{Image Denoising with Total Variation}

A typical application of the primal-dual algorithm would be Total Variation for image denoising. We consider the following total variation problems and their modification with weakly convex terms,
\begin{align}
	& \min_{x\in X} \frac{\lambda}{2}\Vert x-b\Vert^2 +\Vert\nabla x\Vert_1, \tag{Convex model}\\
	& \min_{x\in X} \lambda | \Vert x\Vert^2-b | +\Vert\nabla x\Vert_1, \tag{WC-1} \\
	& \min_{x\in X} | \Vert x\Vert^2- \Vert x_0\Vert^2 | + \frac{\lambda}{2}\Vert x-b\Vert^2 +\Vert\nabla x\Vert_1, \tag{WC-2}\\
	& \min_{x\in X} \Vert x-x_0\Vert_1 + \frac{\lambda}{2}\Vert x-b\Vert^2 +\Vert\nabla x\Vert_1, \tag{WC-3}\\
	& \min_{x\in X} \Vert x^2-x_0^2\Vert_1 + \frac{\lambda}{2}\Vert x-b\Vert^2 +\Vert\nabla x\Vert_1, \tag{WC-4}
\end{align}
where $x^2$ is taken element-wise. All the weakly convex models are sharp except the first model. If we replace $b$ with $\Vert x_0\Vert^2$ in the model (WC-1), then we have no information on the noisy image, which does not make sense. That is why we include $x_0$ and $b$ in all other models. We use Algorithm \ref{alg: prima-dual no with primal update 1st} to compare these different models.

The general problem can be written in the form 
\[
\min_{x\in X} f(x) + \Vert \nabla x\Vert_1,
\]
where $\nabla:\mathbb{R}^{n^2} \to \mathbb{R}^{2n^2}$ is the discrete gradient operator, which can be calculated as $
(\nabla x)_{i,j} = ((\nabla x)^1_{i,j},(\nabla x)^2_{i,j}), 
$
where
\begin{equation*}
	(\nabla x)^1_{i,j} = \begin{cases}
		u_{i+1,j}-u_{i,j} & i<n, \\
		0 & i=n.
	\end{cases},
	\qquad
	(\nabla x)^2_{i,j} = \begin{cases}
		u_{i,j+1}-u_{i,j} & j<n, \\
		0 & j=n.
	\end{cases}
\end{equation*}
We use $g(Lx) = \Vert \nabla x\Vert_1$ as the composite function. The dual problem will have the form 
\[
\max_{y\in \mathbb{R}^{2n^2}} \min_{x\in \mathbb{R}^{n^2}} f(x) - \langle x, \mathrm{div } y \rangle -\delta_P (y),
\]
where $\delta_P(y)$ is the indicator function on the set $P = \{ y\in \mathbb{R}^{2n^2}: \Vert y\Vert_\infty \leq 1 \}$ with
\[
\Vert y\Vert_\infty = \max_{i,j} |y_{i,j}|, \qquad |y_{i,j}| = \sqrt{(y_{i,j}^1)^2 + (y_{i,j}^2)^2}.
\]
The function $\mathrm{div } y$ is the discrete divergence operator. All the explicit forms of the gradient and divergence can be found in \nolinebreak\citep{chambolle2004algorithm,chambolle2011first}, as well as an explicit form for the proximal operator. The proximal calculation for $f$ can be made using the explicit form of the previous examples. The proximal for $\ell_1$ norm is the softmax function which is available on \href{https://proximity-operator.net/multivariatefunctions.html}{(Prox-repository)}.

\paragraph{Setting:} We use the same stepsizes as in the previous example with zeros initialization, and the maximum iteration is  $N=2000$. We use (PSNR) to measure the error with respect to the original image.
The images will be the same as in the previous example, corrupted with additive noise in the form $b= x_0 + 0.1*\mathcal{N}(0,1)$, where $\mathcal{N}(0,1)$ is the matrix with standard normal distribution. The results are illustrated in Figure \ref{fig: Denoising-1}.


\begin{figure}
	\centering
	\includegraphics[width=0.8\textwidth, clip]{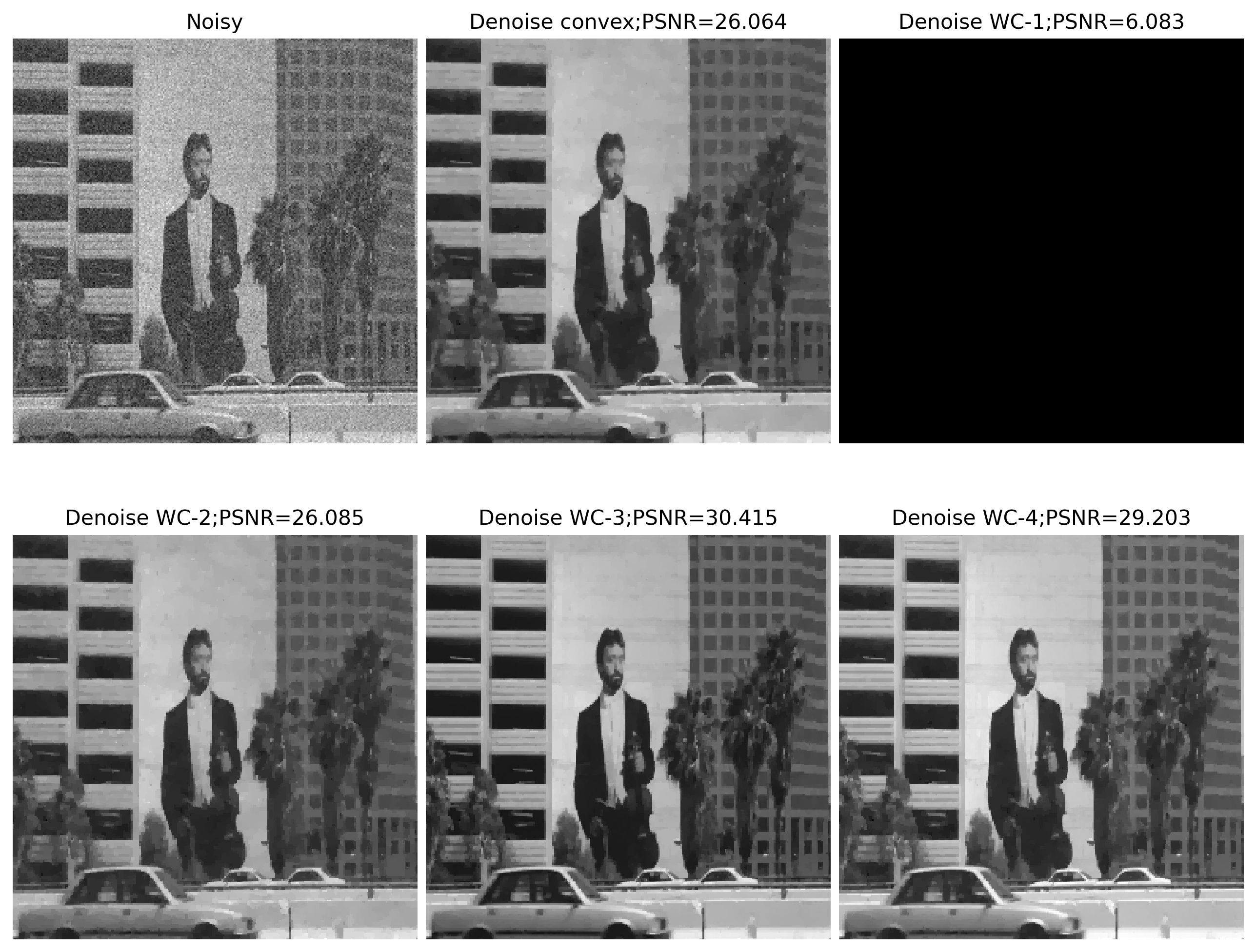}\\
	\includegraphics[width=0.8\textwidth, clip]{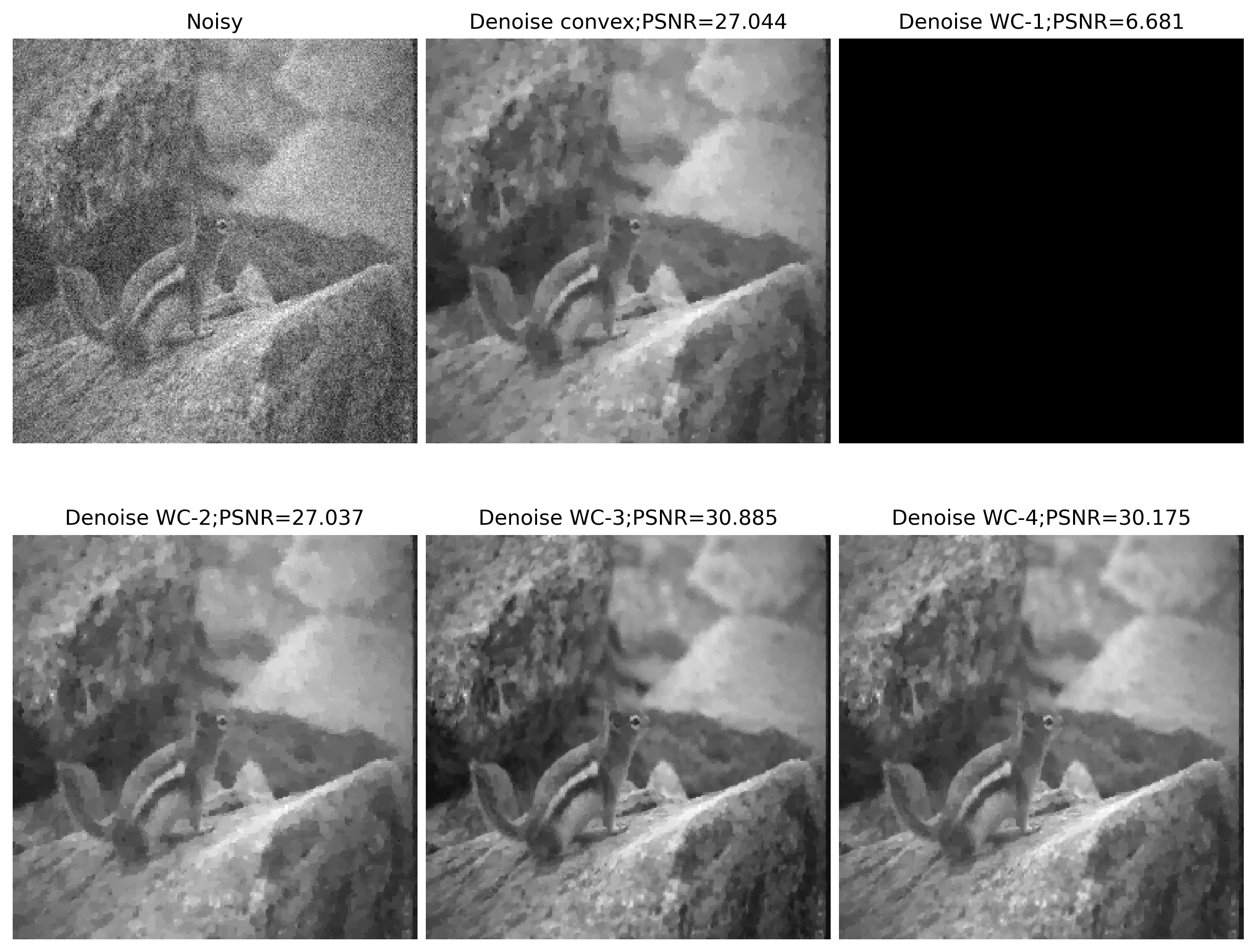} 
	\caption{Image Denoising using total variation.}
	\label{fig: Denoising-1}
\end{figure}

\subsubsection*{Acknowledgments}
\noindent This work was funded by the European Union's Horizon 2020 research and innovation program under the Marie Sk{\l}odowska--Curie grant agreement No 861137. This work represents only the authors' view, and the European Commission is not responsible for any use that may be made of the information it contains.
\subsubsection*{Supplementary Data}
\noindent 
The data used in this work can be found online at \href{https://github.com/majedelhelou/denoising_datasets/tree/main}{Denoising dataset}.

\bibliographystyle{cas-model2-names}
\bibliography{reference}


\end{document}